\documentclass[a4paper]{amsart}
\usepackage{amssymb, enumitem}
\usepackage[all]{xy}
\usepackage{hyperref, aliascnt, graphicx}
\usepackage{mathtools}

\newtheorem{lma}{Lemma}[section]

\newaliascnt{thmCt}{lma}
\newtheorem{thm}[thmCt]{Theorem}
\aliascntresetthe{thmCt}

\newaliascnt{corCt}{lma}
\newtheorem{cor}[corCt]{Corollary}
\aliascntresetthe{corCt}

\newaliascnt{prpCt}{lma}
\newtheorem{prp}[prpCt]{Proposition}
\aliascntresetthe{prpCt}

\theoremstyle{definition}

\newaliascnt{pgrCt}{lma}
\newtheorem{pgr}[pgrCt]{}
\aliascntresetthe{pgrCt}

\newaliascnt{dfnCt}{lma}
\newtheorem{dfn}[dfnCt]{Definition}
\aliascntresetthe{dfnCt}

\newaliascnt{rmkCt}{lma}
\newtheorem{rmk}[rmkCt]{Remark}
\aliascntresetthe{rmkCt}

\newaliascnt{rmksCt}{lma}
\newtheorem{rmks}[rmksCt]{Remarks}
\aliascntresetthe{rmksCt}

\newaliascnt{exaCt}{lma}
\newtheorem{exa}[exaCt]{Example}
\aliascntresetthe{exaCt}

\newaliascnt{qstCt}{lma}
\newtheorem{qst}[qstCt]{Question}
\aliascntresetthe{qstCt}

\newaliascnt{pbmCt}{lma}
\newtheorem{pbm}[pbmCt]{Problem}
\aliascntresetthe{pbmCt}

\newaliascnt{ntnCt}{lma}
\newtheorem{ntn}[ntnCt]{Notation}
\aliascntresetthe{ntnCt}

\def\today{\number\day\space\ifcase\month\or   January\or February\or
   March\or April\or May\or June\or   July\or August\or September\or
   October\or November\or December\fi\   \number\year}

\newcommand{\ZZ}{{\mathbb{Z}}}
\newcommand{\NN}{{\mathbb{N}}}
\newcommand{\KK}{{\mathbb{K}}}
\newcommand{\CC}{{\mathbb{C}}}
\newcommand{\Prim}{{\mathrm{Prim}}}
\newcommand{\ca}{$C^*$-algebra}
\newcommand{\stHom}{$\ast$-homomorphism}
\newcommand{\andSep}{\,\,\,\text{ and }\,\,\,}
\newcommand{\axiomO}[1]{(O#1)}
\newcommand{\CatCu}{\ensuremath{\mathrm{Cu}}}
\newcommand{\CuSgp}{$\CatCu$-sem\-i\-group}
\newcommand{\CuMor}{$\CatCu$-mor\-phism}
\newcommand{\llideal}{\ensuremath{<\hskip-7pt\lhd\hskip3pt{}}}

\DeclareMathOperator{\linSpan}{span}
\DeclareMathOperator{\Cu}{Cu}
\DeclareMathOperator{\Lat}{Lat}

\newcounter{theoremintro}

\newtheorem{prpIntro}[theoremintro]{Proposition}
\newtheorem{pbmIntro}[theoremintro]{Problem}
\newtheorem{thmIntro}[theoremintro]{Theorem}

\title{The Global Glimm Property}
\date{\today}

\author{Hannes Thiel}
\address{Hannes Thiel,
Department of Mathematics, Kiel University, Heinrich-Hecht-Platz~6, 24118 Kiel, Germany.}
\email{hannes.thiel@math.uni-kiel.de}
\urladdr{www.hannesthiel.org}

\author{Eduard Vilalta}
\address{Eduard Vilalta,
Departament de Matem\`{a}tiques,
Universitat Aut\`{o}noma de Barcelona,
08193 Bellaterra, Barcelona, Spain}
\email{evilalta@mat.uab.cat}
\urladdr{www.eduardvilalta.com}

\thanks{The first named author was partially supported by the Deutsche Forschungsgemeinschaft (DFG, German Research Foundation) under Germany's Excellence Strategy EXC 2044-390685587 (Mathematics M\"{u}nster: Dynamics-Geometry-Structure) and by the ERC Consolidator Grant No.~681207.
The second named author was partially supported by MINECO (grant No.\ PRE2018-083419 and No.\ PID2020-113047GB-I00), and by the Comissionat per Universitats i Recerca de la Generalitat de Catalunya (grant No.\ 2017SGR01725).
}

\subjclass[2010]%
{Primary
46L05; 
Secondary
19K14, 
46L80, 
46L85. 
}
\date{\today}

\begin{document}

\begin{abstract}
It is known that a \ca{} with the Global Glimm Property is nowhere scattered (it has no elementary ideal-quotients), and the Global Glimm Problem asks if the converse holds.
We provide a new approach to this long-standing problem by showing that a \ca{} has the Global Glimm Property if and only if it is nowhere scattered and its Cuntz semigroup is \emph{ideal-filtered} (the Cuntz classes generating a given ideal are downward directed) and has \emph{property (V)} (a weak form of being sup-semilattice ordered).

We show that ideal-filteredness and property~(V) are automatic for \ca{s} that have stable rank one or real rank zero, thereby recovering the solutions to the Global Glimm Problem in these cases.
We also use our approach to solve the Global Glimm Problem for new classes of \ca{s}.
\end{abstract}

\maketitle

\section{Introduction}

A classical result of Glimm shows that a \ca{} $A$ admits a nonzero \mbox{$\ast$-homo}\-morphism $C_0((0,1],M_k)\to A$ if and only if it has an irreducible representation of dimension at least $k$;
see \cite{Gli61Type1} and \cite[Proposition~3.10]{RobRor13Divisibility}.
It follows that all irreducible representations of $A$ have dimension at least $k$ if and only if there exists a family of *-homomorphisms $C_0((0,1],M_k)\to A$ whose combined images generate~$A$ as a closed, two-sided ideal.

Going one step further, one can show that no hereditary sub-\ca{} of $A$ admits a finite-dimensional irreducible representation if and only if for each \mbox{$a\in A_+$}, each natural number $k\geq 2$, and each $\varepsilon>0$, there exist finitely many $\ast$-ho\-mo\-mor\-phisms $\varphi_1,\ldots,\varphi_n\colon C_0((0,1],M_k)\to\overline{aAa}$ such that $(a-\varepsilon)_+$ belongs to the ideal of $A$ generated by the combined images of $\varphi_1,\ldots,\varphi_n$. Such algebras were termed \emph{nowhere scattered} in \cite{ThiVil21arX:NowhereScattered}.

If this can always be achieved with a single $\ast$-homomorphism (that is, $n=1$), then $A$ is said to have the \emph{Global Glimm Property}; 
see \autoref{dfn:GGP}.
This property was first introduced by Kirchberg and R{\o}rdam in their study of (weakly) purely infinite \ca{s}, \cite[Definition~4.12]{KirRor02InfNonSimpleCalgAbsOInfty}, and one can show that the algebras satisfying the Global Glimm Property have many interesting characteristics. 
For example, a unital \ca{} with this property always contains a full square-zero element, \autoref{prp:GlimmFirstChar}, and the existence of such elements has considerable implications on the structure of the unitary group of the \ca{}; 
see \cite{ChaRob21arX:AutoContGrUnit}.

Clearly, every \ca{} that has the Global Glimm Property is nowhere scattered. 
The \emph{Global Glimm Problem} asks if the converse holds. 
That is, it asks if every nowhere scattered \ca{} has the Global Glimm Property.

This problem was raised implicitly by Kirchberg and R{\o}rdam in \cite{KirRor02InfNonSimpleCalgAbsOInfty} (see Section~4, and the discussion after Question~9.5 in \cite{KirRor02InfNonSimpleCalgAbsOInfty}).
Indeed, a positive solution to the Global Glimm Problem would imply that every weakly purely infinite \ca{} is purely infinite, a long-standing question first posed in \cite{KirRor02InfNonSimpleCalgAbsOInfty}.
The Global Glimm Problem was also implicitly studied by Blanchard and Kirchberg in \cite{BlaKir04GlimmHalving, BlaKir04PureInf}, and mentioned explicitly by Elliott and R{\o}rdam in \cite[Question~1.2]{EllRor06Perturb};
see also \cite[Remark~5.7]{RobRor13Divisibility}.

The Global Glimm Problem has been solved affirmatively in a variety of situations: 
It was already noted in \cite[Lemma~4.14]{KirRor02InfNonSimpleCalgAbsOInfty} that the problem has a positive solution whenever the \ca{} under study is either simple, approximately divisible, or purely infinite. 
Later, positive solutions were obtained for \ca{s} with a Hausdorff, finite-dimensional primitive ideal space (\cite{BlaKir04GlimmHalving}), for \ca{s} of real rank zero (\cite{EllRor06Perturb}), and recently for \ca{s} of stable rank one (\cite{AntPerRobThi22CuntzSR1}).
 
In this paper, we study the problem by first characterizing the Global Glimm Property and nowhere scatteredness through divisibility properties (recalled in \autoref{pgr:2OmegaDef}) of the Cuntz semigroup of the \ca.
We show that a \ca{} has the Global Glimm Property if and only if its Cuntz semigroup is $(2,\omega)$-divisible;
see \autoref{prp:GlimmFirstChar}.
On the other hand, by \cite[Theorem~8.9]{ThiVil21arX:NowhereScattered}, a \ca{} is nowhere scattered if and only if its Cuntz semigroup is weakly $(2,\omega )$-divisible; 
see \autoref{pgr:2OmegaDef}. 
Thus, one can reformulate the Global Glimm Problem as: 
Is every weakly $(2,\omega)$-divisible Cuntz semigroup $(2,\omega)$-divisible?
(See \autoref{pbm:ReformulateDiv}.)

The Cuntz semigroup of a \ca{} is an object of the category $\CatCu$ of abstract Cuntz semigroups, also called \CuSgp{s}, introduced in \cite{CowEllIva08CuInv} and extensively studied in \cite{AntPerThi18TensorProdCu, AntPerThi20AbsBivariantCu, AntPerThi20AbsBivarII, AntPerThi20CuntzUltraproducts} and \cite{AntPerRobThi21Edwards}.
We are thus led to study conditions that characterize when a weakly $(2,\omega)$-divisible \CuSgp{} is $(2,\omega)$-divisible.

The main accomplishment of this paper is the identification of two conditions on a Cuntz semigroup that precisely capture when weak $(2,\omega)$-divisibility implies $(2,\omega)$-divisibility:
The first condition, termed \emph{ideal-filteredness}, roughly means that the Cuntz classes generating a given ideal are downward directed;
see \autoref{dfn:IdealFiltered} and \autoref{rmk:IdealFiltered}.
The second condition, termed \emph{property~(V)}, is a (very weak) form of being sup-semilattice ordered;
see \autoref{dfn:PrpV} and \autoref{rmk:PrpV}.

The main technical result of this paper is:

\begin{thmIntro}[\ref{prp:MainThmCu}]
A \CuSgp{} satisfying \axiomO{5}-\axiomO{8} is $(2,\omega)$-divisible if and only if it is weakly $(2,\omega)$-divisible, ideal-filtered, and has property~(V).
\end{thmIntro}

Here, \axiomO{5}-\axiomO{8} are conditions on a \CuSgp{} that automatically hold for Cuntz semigroups of \ca{s}.
We thus obtain the main result of the paper:

\begin{thmIntro}[\ref{prp:CharGGP}]
Let $A$ be a \ca{}. 
Then the following are equivalent:
\begin{enumerate}
\item 
$A$ has the Global Glimm Property;
\item
$\Cu(A)$ is $(2,\omega)$-divisible;
\item 
$\Cu(A)$ is weakly $(2,\omega)$-divisible, ideal-filtered and has property~(V);
\item 
$A$ is nowhere scattered and $\Cu(A)$ is ideal-filtered and has property~(V).
\end{enumerate}
\end{thmIntro}

This leads to the following reformulation of the Global Glimm Problem:

\begin{pbmIntro}[\ref{pbm:AutomaticIFV}]
Let $A$ be a nowhere scattered \ca.
Is $\Cu(A)$ ideal-filtered and has property~(V)?
\end{pbmIntro}

In \autoref{sec:IdealFiltered}, we introduce and study ideal-filteredness. 
We show that this condition is automatically satisfied for Cuntz semigroup of \ca{s} that have stable rank one  (\autoref{prp:IdealFilteredSR1}), or real rank zero (\autoref{prp:IdealFilteredRR0}), or are separable and have topological dimension zero (\autoref{prp:IdealFilteredTopDimZero}).

In \autoref{sec:V}, we introduce and study property~(V), and we show that this condition automatically holds for the Cuntz semigroups of \ca{s} that are residually stably finite (\autoref{prp:V-RSF}), have real rank zero (\autoref{prp:V-rr0}), or are separable, have topological dimension zero and strict comparison of positive elements (\autoref{prp:V-AlmUnp}).

In particular, Cuntz semigroups of \ca{s} of stable rank one or real rank zero are automatically ideal-filtered and have property~(V) (even if the \ca{} is not nowhere scattered), and we recover the known positive solutions of the Global Glimm Problem in these cases:

\begin{prpIntro}[\ref{prp:GGP-SR1}, \ref{prp:GGP-RR0}]
Let $A$ be a \ca{} of stable rank one or real rank zero.
Then $A$ has the Global Glimm Property if and only if $A$ is nowhere scattered.
\end{prpIntro}

We also obtain a positive solution to the Global Glimm Problem for \ca{s} that were not covered by previous results:

\begin{prpIntro}[\ref{prp:GGP-TopDimZero}]
Let $A$ be a separable \ca{} with topological dimension zero. 
Assume that $A$ is residually stably finite or has strict comparison of positive elements.
Then $A$ has the Global Glimm Property if and only if $A$ is nowhere scattered.
\end{prpIntro}

\autoref{exa:NotIdealFiltered} shows that ideal-filteredness does not hold for the Cuntz semigroup of every \ca{}.
On the other hand, we do not know of any \ca{} whose Cuntz semigroup does not have property~(V);
see \autoref{qst:V}.

In \autoref{sec:Glimm} we also study permanence properties of the Global Glimm Property.
We show that it passes to hereditary sub-\ca{s}, quotients, and extensions (\autoref{prp:GlimmExtension}), to inductive limits (\autoref{prp:Limits}), and is invariant under Morita equivalence (\autoref{prp:Morita}).
Using this, we prove:

\begin{thmIntro}[\ref{prp:largestGlimmIdeal}]
Every \ca{} contains a (unique) largest ideal that has the Global Glimm Property.
\end{thmIntro}

In forthcoming work, \cite{ThiVil22pre:Soft}, we will give other characterizations of the Global Glimm Property of a \ca{} through the abundance of soft elements.

\subsection*{Acknowledgements}
The authors want to thank the referees for their careful reading of the manuscript and their valuable suggestions.

\section{Preliminaries}\label{sec:Prelims}

We use $A_+$ to denote the set of positive elements in a \ca{} $A$.
Given $a,b\in A_+$, we write $a\precsim b$ if $a=\lim_n r_n b r_n^*$ for some sequence $(r_n)_n$ in $A$.
This induces an equivalence relation $\sim$ on $A_+$ by setting $a\sim b$ if $a\precsim b$ and $b\precsim a$.
Both of these relations were introduced by Cuntz in \cite{Cun78DimFct}.

The \emph{Cuntz semigroup} of $A$, denoted by $\Cu(A)$, is defined as $(A\otimes\mathcal{K})_+/\sim$.
The relation $\precsim$ induces a partial order on $\Cu(A)$, and addition of orthogonal positive elements induces an abelian monoid structure on $\Cu(A)$.
This turns $\Cu(A)$ into a positively ordered monoid.
We refer to \cite[Section~3.2]{AntPerThi18TensorProdCu} for details.
 
Let $P$ be a partially ordered set where every increasing sequence has a supremum.
Given $x,y\in P$, we write $x\ll y$, and we say that $x$ is \emph{way-below} $y$, if for every increasing sequence $(z_n)_n$ such that $y\leq \sup_n z_n$ there exists $m$ such that $x\leq z_m$.

It was shown in \cite{CowEllIva08CuInv} that the Cuntz semigroup of every \ca{} has the following properties:
\begin{itemize}
\item[\axiomO{1}] 
Every increasing sequence has a supremum.
\item[\axiomO{2}] 
Every element is the supremum of a $\ll$-increasing sequence.
\item[\axiomO{3}]
If $x'\ll x$ and $y'\ll y$, then $x'+y'\ll x+y$.
\item[\axiomO{4}] 
Given increasing sequences $(x_n)_n$ and $(y_n)_n$, we have $\sup_n(x_n+y_n)=\sup_n x_n + \sup_n y_n$.
\end{itemize}

An abstract Cuntz semigroup, or \emph{\CuSgp{}}, is defined as a positively ordered monoid satisfying \axiomO{1}-\axiomO{4}. 
A \emph{\CuMor{}} is an order-preserving monoid morphism that also preserves suprema of increasing sequences and the way-below relation.
The resulting category $\CatCu$ of \CuSgp{s} and \CuMor{s} is bicomplete (it has arbitrary limits and colimits);
see \cite{AntPerThi20CuntzUltraproducts}. Further, every \stHom{} $\varphi\colon A\to B$ induces a \CuMor{} $\Cu (\varphi )\colon \Cu(A)\to \Cu (B)$, and the functor $\Cu\colon C^*\to \CatCu$ is continuous; see \cite[Corollary~3.2.9]{AntPerThi18TensorProdCu}.

Since the groundbreaking work in \cite{CowEllIva08CuInv}, further properties that the Cuntz semigroup of every \ca{} satisfies have been uncovered:
\begin{itemize}
\item[\axiomO{5}] 
Given $x'\ll x$, $y'\ll y$ and $x+y\leq z$, there exists $c$ such that $y'\ll c$ and $x' +c \leq z \leq x+c$.
\item[\axiomO{6}] 
Given $x'\ll x\leq y+z$, there exist $v$ and $w$ such that $v\ll  x,y$, and $w\ll x,z$ satisfying $x'\ll v +w$.
\item[\axiomO{7}] 
Given $x_1'\ll x_1\ll w$ and $x_2'\ll x_2\ll w$, there exists $x\in S$ such that
\[
x_1',x_2'\ll x \ll w,x_1+x_2.
\]
\end{itemize}

Property~\axiomO{5} was shown in \cite[Proposition~4.6]{AntPerThi18TensorProdCu}.
This property is often applied for the case $y=0$, in which case it says:
Given $x'\ll x \leq z$, there exists~$c$ such that $x'+c\leq z\leq x+c$.
(This special case of \axiomO{5} had already been proven in \cite[Lemma~7.1]{RorWin10ZRevisited}.)

Properties \axiomO{6} and \axiomO{7} for Cuntz semigroups were proven in \cite[Proposition~5.1.1]{Rob13Cone} and \cite[Proposition~2.2]{AntPerRobThi21Edwards}.
We note that the version of \axiomO{7} given here is slightly different from, but easily seen to be equivalent to, \cite[Definition~2.1]{AntPerRobThi21Edwards}. Recently, a new property that every Cuntz semigroup also satisfies, termed \axiomO{8}, has been introduced in \cite[Definition~7.2]{ThiVil21arX:NowhereScattered}:
\begin{itemize}
\item[\axiomO{8}] Given $x'\ll x$, $y'\ll y$, $w=2w$ and $z$ in $ S$ such that $x+y\ll z+w$, there exist $z_1,z_2\in S$ satisfying
\[
z_1+z_2\ll z, \quad
x'\ll z_1+w, \quad
y'\ll z_2+w, \quad
z_1\ll x+w, \andSep
z_2\ll y+w.
\] 
\end{itemize}

The next result easily follows by induction from \axiomO{7}:
\begin{lma}
\label{prp:RefO7}
Let $S$ be a \CuSgp{} satisfying \axiomO{7}, and let $x_j',x_j,w\in S$ satisfy
\[
x_j'\ll x_j\leq w
\]
for $j=1,\ldots,n$.
Then there exists $x\in S$ such that
\[
x_1',\ldots,x_n'\ll x \ll w, x_1+\ldots+x_n.
\]
\end{lma}

We will also repeatedly make use of the following consequence of \axiomO{5}. We add a proof for the convienience of the reader.
\begin{lma}\label{prp:DivO5}
 Let $S$ be a \CuSgp{} satisfying \axiomO{5}, and let $k\in\NN$. Then, whenever $x',x,z\in S$ are such that $x'\ll x$ and $kx\leq z$, there exists $y\in S$ satisfying 
 \[
  (k-1)x'+y\leq z\leq ky,\andSep 
  x'\ll y.
 \]
\end{lma}
\begin{proof}
 Let $x',x,z\in S$ be as above, and take $x^\natural\in S$ such that $x'\ll x^\natural \ll x$. Then, we have 
 \[
  (k-1)x'\ll (k-1)x^\natural ,\quad 
  x^\natural \ll x,\andSep 
  (k-1)x^\natural + x \leq z.
 \]

 Applying \axiomO{5}, we obtain an element $y$ satisfying 
 \[
  x^\natural \ll y,\andSep 
  (k-1)x' + y\leq z\leq (k-1)x^\natural +y.
 \]

 It is readily checked that $y$ has the required conditions.
\end{proof}

\begin{pgr}
A submonoid of a \CuSgp{} is an \emph{ideal} if it is downward-hereditary and closed under suprema of increasing sequences.
Given a \CuSgp{} $S$ and an ideal $I\subseteq S$, the binary relations $\leq_I$ and $\sim_I$ on $S$ are defined by setting $x\leq_I y$ if $x\leq y+z$ for some $z\in I$, and by writing $x\sim_I y$ whenever $x\leq_I y$ and $y\leq_I x$.
One can check that $\sim_I$ is an equivalence relation, and that the \emph{quotient} $S/\sim_I$, which is usually denoted by $S/I$, is a \CuSgp{} when endowed with the order induced by $\leq_I$ and the expected addition;
see \cite[Lemma~5.1.2]{AntPerThi18TensorProdCu}.

Given a \ca{} $A$ and a closed, two-sided ideal $I$ of $A$, it follows that $\Cu(I)$ naturally is an ideal of $\Cu(A)$, and we have $\Cu(A/I)\cong \Cu(A)/\Cu(I)$ by \cite[Proposition~3.3]{CiuRobSan10CuIdealsQuot}. 
Moreover, by associating to a closed, two-sided ideal $I\subseteq A$ the ideal $\Cu(I)\subseteq\Cu(A)$, we obtain a natural isomorphism between the complete lattices of (closed, two-sided) ideals in $A$ and $\Cu(A)$;
see \cite[Proposition~5.1.10]{AntPerThi18TensorProdCu}.
\end{pgr}

\begin{pgr}
\label{pgr:2OmegaDef}
Let $S$ be a \CuSgp{}. 
An element $x\in S$ is said to be \emph{$(k,\omega)$-divisible} if for every $x'\in S$ satisfying $x'\ll x$ there exists an element $z$ such that $kz\ll x$ and $x'\ll\infty z$.
Equivalently, for every $x'$ with $x'\ll x$ there exists $z$ such that $kz\leq x$ and $x'\leq\infty z$;
see \cite[Definition~5.1]{RobRor13Divisibility} and \cite[Paragraph~5.1]{AntPerRobThi22CuntzSR1}.

Further, one says that $x$ is \emph{weakly $(k,\omega)$-divisible} if for every $x'\in S$ satisfying $x'\ll x$ there exist elements $z_1,\ldots, z_n$ such that $kz_j\ll x$ for every $j$ and $x'\ll z_1+\ldots +z_n$.
Equivalently, for every $x'$ with $x'\ll x$, there exist $z_1,\ldots, z_n$ such that $kz_j\leq x$ for every $j$ and $x'\leq z_1+\ldots +z_n$.
 
A \CuSgp{} is said to be \emph{(weakly) $(k,\omega )$-divisible} if each of its elements is (weakly) $(k,\omega )$-divisible. By \cite[Theorem~8.9]{ThiVil21arX:NowhereScattered}, a \ca{} is \emph{nowhere scattered} (that is, it has no nonzero elementary ideal-quotients) if and only if its Cuntz semigroup is weakly $(2,\omega )$-divisible.

As proven in \cite[Lemma~8.7]{ThiVil21arX:NowhereScattered}, a weakly $(2,\omega)$-divisible \CuSgp{} is weakly $(k,\omega)$-divisible for every $k\in\NN$.
With a similar argument, one can prove the analogous statement for $(2,\omega)$-divisibility.
\end{pgr}

\section{The Global Glimm Property and Divisibility of the Cuntz semigroup}
\label{sec:Glimm}

In this section we characterize the Global Glimm Property of a \ca{} in various ways, most importantly by the property that its Cuntz semigroup is $(2,\omega)$-divisible;
see \autoref{prp:GlimmFirstChar}.

Thus, since a \ca{} is nowhere scattered if and only if its Cuntz semigroup is weakly $(2,\omega)$-divisible (\cite[Theorem~8.9]{ThiVil21arX:NowhereScattered}), the Global Glimm Problem can be restated as: Is every weakly $(2,\omega)$-divisible Cuntz semigroup $(2,\omega)$-divisible?

\begin{dfn}[{\cite[Definition~4.12]{KirRor02InfNonSimpleCalgAbsOInfty}}]
\label{dfn:GGP}
A \ca{} $A$ has the \emph{Global Glimm Property} if for each $a\in A_+$, each natural number $k\geq 2$, and each $\varepsilon>0$, there exists a \stHom{} $\varphi\colon C_0((0,1],M_k)\to\overline{aAa}$ such that $(a-\varepsilon)_+$ belongs to the ideal of $A$ generated by the image of $\varphi$. 
\end{dfn}

\begin{rmk}
The literature contains several (sometimes nonequivalent) statements of the `Global Glimm (Halving) Property'. 
For example, in \cite[Remark~5.7]{RobRor13Divisibility}, a unital \ca{} $A$ is said to have the `Global Glimm Halving Property' if there exists a \stHom{} $M_2(C_0((0,1]))\to A$ with full image.
The example $A=M_2(\CC)$ shows that this is not equivalent to the `Global Glimm Property' as defined in \cite[Definition~4.12]{KirRor02InfNonSimpleCalgAbsOInfty}.

Further, in \cite[Definition~1.2]{BlaKir04GlimmHalving}, a \ca{} $A$ is said to have the \emph{Global Glimm Halving Property} if for every $a\in A_+$ and $\varepsilon>0$ there exists $r\in\overline{aAa}$ such that $r^2=0$ and such that $(a-\varepsilon)_+$ belongs to the ideal of $A$ generated by~$r$.
We will see in \autoref{prp:GlimmFirstChar} below that this definition is equivalent to \cite[Definition~4.12]{KirRor02InfNonSimpleCalgAbsOInfty}.
\end{rmk}

In preparation for the proof of \autoref{prp:GlimmFirstChar}, we need some basic results about scales.
A \emph{scale} in a \CuSgp{} $S$ is a subset $\Sigma\subseteq S$ that is downward-hereditary, closed under suprema of increasing sequences, and that generates $S$ as an ideal;
see \cite[Definition~4.1]{AntPerThi20CuntzUltraproducts}. 
The \emph{scale} of a \ca{} $A$ was defined in \cite[4.2]{AntPerThi20CuntzUltraproducts} as:
\[
\Sigma_A := \big\{ x\in\Cu(A): \text{for each $x'\ll x$ there exists $a\in A_+$ with $x'\leq[a]$} \big\},
\]

The next result provides basic properties about $\Sigma_A$.

\begin{lma}
\label{prp:ScaleCuA}
Let $A$ be a \ca.
Then the following statements hold:
\begin{enumerate}
\item
Let $x\in\Sigma_A$ and $x'\in\Cu(A)$ satisfy $x'\ll x$.
Then there exists $a\in A_+$ such that $x'\ll[a]\ll x$.
\item
We have
\[
\Sigma_A = \big\{ x\in\Cu(A): \text{there exists $a\in A_+$ with $x\leq[a]$} \big\}.
\]
\item
The set $\Sigma_A$ is upward directed:
given $x,y\in\Sigma_A$ there exists $z\in\Sigma_A$ such that $x,y\leq z$.
\item
Let $x\in\Cu(A)$.
Then there exists $a\in A_+$ such that $x\leq\infty[a]$.
\item
If $A$ is $\sigma$-unital with strictly positive element $h\in A$, then $\Sigma_A=\{x\in\Cu(A): x\leq[h]\}$.
\end{enumerate}
\end{lma}
\begin{proof}

(1) 
Choose $x''\in\Cu(A)$ such that $x'\ll x''\ll x$.
By definition of $\Sigma_A$, we obtain $b\in A_+$ such that $x''\leq[b]$.
Choose $c\in(A\otimes\KK)_+$ such that $x''=[c]$.
Then choose $\varepsilon>0$ such that $x'\ll[(c-\varepsilon)_+]$.
Since $c\precsim b$, we can apply R\o{}rdam's Lemma (see, for example, \cite[Theorem~2.30]{Thi17:CuLectureNotes}), to obtain $r\in A\otimes\KK$ such that $(c-\varepsilon)_+=r^*r$ and $rr^*\in\overline{bAb}$.
Set $a:=rr^*$, which belongs to $A_+$.
We have
\[
x'\ll[(c-\varepsilon)_+]=[r^*r]=[rr^*]=[a], \andSep
[a]=[(c-\varepsilon)_+]\ll[c]=x''\ll x.
\]

(2) 
Let $x\in\Sigma_A$.
Choose a $\ll$-increasing sequence $(x_n)_n$ in $\Cu(A)$ with supremum~$x$.
For each $n$, we obtain $a_n\in A_+$ such that $x_n\leq[a_n]$.
Set
\[
a := \sum_{n=1}^\infty \frac{1}{2^n\|a_n\|}a_n \in A_+.
\]

Then, for each $n$, we have
\[
x_n \leq [a_n] \leq [a],
\]
and therefore $x\leq[a]$, as desired.

(3)  
Let $x_1,x_2\in\Sigma_A$.
Applying~(2), we obtain $a_1,a_2\in A_+$ such that $x_1\leq[a_1]$ and $x_2\leq[a_2]$.
Set $a:=a_1+a_2\in A_+$.
Then $[a]\in\Sigma_A$ is the desired upper bound for $x_1$ and $x_2$.

(4) 
Since $\Sigma_A$ generates $\Cu(A)$ as an ideal, we obtain a sequence $(y_n)_n$ in $\Sigma_A$ such that $x\leq\sum_{n=1}^\infty y_n$.
Applying~(3), and using that $\Sigma_A$ is closed under suprema of increasing sequences, we obtain $y\in\Sigma_A$ such that $y_n\leq y$ for each $n$.
Applying~(2), we obtain $a\in A_+$ such that $y\leq[a]$. 
It follows that
\[
x 
\leq \sum_{n=1}^\infty y_n
\leq \sum_{n=1}^\infty y
= \infty y
\leq \infty[a].
\]

(5) 
Assume that $A$ is $\sigma$-unital and that $h\in A_+$ is a strictly positive element. 
Then, we know by construction that $\{x\in\Cu(A): x\leq[h]\}\subseteq \Sigma_A$. 
Further, since $h$ is strictly positive, we have $a\precsim h$ for every $a\in A_+$.

Thus, given $x\in \Sigma_A$, we know by (2) that $x\leq [a]$ for some $a\in A_+$. 
Since $a\precsim h$, we get $x\leq [h]$. This shows $\Sigma_A\subseteq \{x\in\Cu(A): x\leq[h]\}$.
\end{proof}

As mentioned in \autoref{pgr:2OmegaDef}, a $(2,\omega)$-divisible \CuSgp{} is $(k,\omega)$-divisible for every $k\in\NN$. 
We will also use a version for scales:

\begin{lma}
\label{prp:Div2DivN}
Let $\Sigma$ be a scale in a \CuSgp{} $S$.
Assume that every element in~$\Sigma$ is $(2,\omega)$-divisible.
Then every element in~$\Sigma$ is $(k,\omega)$-divisible for every $k\geq 2$.
\end{lma}
\begin{proof}
 Let $k\in \NN$ and take $n\in\NN$ such that $k\leq 2^n$. Given $x',x\in \Sigma$ such that $x'\ll x$,  there exists $y_1\in S$ such that $2y_1\ll x$ and $x'\ll \infty y_1$. Since $y_1$ is bounded by $x$ and $\Sigma$ is downward-hereditary, we get $y_1\in \Sigma$.
 
 Now take $y_1'\ll y_1$ such that $x'\ll \infty y_1'$. Applying $(2,\omega )$-divisibility to $y_1'\ll y_1$, we obtain $y_2\in S$ satisfying $y_1'\ll\infty y_2$ and $2y_2\ll y_1$. Since $y_2$ is bounded by $y_1\in\Sigma$, we get $y_2\in\Sigma$. Note that $4y_2\ll 2y_1\ll x$ and $x'\ll \infty y_1'\leq \infty y_2$.
 
 Repeating this argument $n$ times, we obtain $y_n\in\Sigma$ such that $x'\ll \infty y_n$ and $ky_n\leq 2^n y_n\ll x$, as desired.
\end{proof}

\begin{lma}
\label{prp:ScaleDiv}
Let $\Sigma$ be a scale in a \CuSgp{} $S$ satisfying \axiomO{5}, \axiomO{6} and \axiomO{7}.
Assume that every element in $\Sigma$ is $(2,\omega)$-divisible.
Then every element in $S$ is $(2,\omega)$-divisible.
\end{lma}
\begin{proof}
For $d\geq 1$, we consider the $d$-fold amplification
\[
\Sigma^{(d)} 
:= \big\{ x\in S : \text{ for each $x'\ll x$ there are $y_1,\ldots,y_d\in\Sigma$ with $x'\leq \sum_{j=1}^d y_j$} \big\}.
\]

It suffices to show that every element in $\Sigma^{(2)}$ is $(2,\omega)$-divisible, for then an inductive argument shows that every element in $\Sigma^{(2^n)}$ is $(2,\omega)$-divisible for every $n$, which implies the desired result. Indeed, for any pair of elements $x',x\in S$ such that $x'\ll x$ there exist $n\in\NN$ and elements $y',y\in \Sigma^{(2^n)}$ satisfying $x'\ll y'\ll y\ll x$. Thus, if $y\in \Sigma^{(2^n)}$ is $(2,\omega)$-divisible, one can find $z\in S$ such that $2z\ll y$ and $y'\ll \infty z$. This would imply $2z\ll x$ and $x'\ll \infty z$, as required.

So let $x\in \Sigma^{(2)}$ and $x''\in S$ satisfy $x''\ll x$.
Choose $x'$ such that $x''\ll x'\ll x$, and let  $y_1,y_2\in\Sigma$ be such that $x'\leq y_1+y_2$.
Applying \axiomO{6}, we obtain $z_1,z_2\in S$ satisfying
\[
x''\ll z_1+z_2, \andSep z_j\leq x',y_j \quad \text{for $j=1,2$}.
\]

Choose $z_1',z_2'$ such that
\[
x''\ll z_1'+z_2', \andSep z_j'\ll z_j \quad \text{ for $j=1,2$}.
\]

We have $z_j\leq y_j\in\Sigma$ and therefore $z_j\in\Sigma$ for $j=1,2$.
Hence, by \autoref{prp:Div2DivN}, $z_1$ and $z_2$ are $(3,\omega)$-divisible and we obtain $c$ and $d$ such that
\[
3c\leq z_1, \quad
z_1'\ll \infty c, \quad
3d\leq z_2, \andSep
z_2'\ll \infty d.
\]

Choose $c'$ and $d'$ such that
\[
c'\ll c, \quad
z_1'\ll \infty c', \quad
d'\ll d, \andSep
z_2'\ll\infty d'.
\]

Since $z_1\leq x$, we can apply \autoref{prp:DivO5} to $3c\leq x$ and obtain $e$ satisfying
\[
2c'+e\leq x\leq 3e, \andSep c'\ll e.
\]

Then $d'\ll d\leq z_2\leq x\leq 3e$.
Applying \axiomO{6}, we obtain $e_1,e_2,e_3$ such that
\[
d'\ll e_1+e_2+e_3, \andSep e_j\leq d,e \quad \text{for $j=1,2,3$}.
\]

Choose $e_1',e_2',e_3'$ such that
\[
d'\ll e_1'+e_2'+e_3', \andSep e_j'\ll e_j \quad \text{ for $j=1,2,3$}.
\]

Applying \axiomO{7} for $e_j'\ll e_j\leq e$, we obtain $f$ such that
\[
e_1',e_2',e_3'\leq f \leq e, e_1+e_2+e_3.
\]

Then
\[
z_2' \ll \infty d' \leq \infty (e_1'+e_2'+e_3') \leq \infty f.
\]

Further, we have
\[
f\leq e_1+e_2+e_3 \leq 3d \leq z_2.
\]

Therefore, we have $f\in\Sigma$.
Choose $f'$ such that
\[
f'\ll f, \andSep z_2'\ll \infty f'.
\]

Using that $f$ is $(2,\omega)$-divisible, we obtain $g$ satisfying
\[
2g\leq f, \andSep f'\ll\infty g.
\]

We claim that $h:=c'+g$ has the desired properties.
Using that $2g\leq f\leq e$, we have
\[
2h
= 2c' + 2g
\leq 2c'+e
\leq x.
\]

Further, we have $z_2'\ll \infty f' \leq \infty g$, and thus
\[
x'' 
\ll z_1'+z_2'
\leq \infty c' + \infty g
= \infty h,
\]
as desired.
\end{proof}

\begin{thm}
\label{prp:GlimmFirstChar}
Let $A$ be a \ca.
Then the following are equivalent:
\begin{enumerate}
\item
$A$ has the Global Glimm Property:
for every $a\in A_+$, every $k\geq 2$, and every $\varepsilon>0$, there is a \stHom{} $\varphi\colon M_k(C_0((0,1]))\to\overline{aAa}$ such that $(a-\varepsilon)_+$ belongs to the ideal generated by the image of $\varphi$;
\item
for every $a\in A_+$, and every $\varepsilon>0$, there is a \stHom{} of the form $\varphi\colon M_2(C_0((0,1]))\to\overline{aAa}$ such that $(a-\varepsilon)_+$ belongs to the ideal generated by the image of $\varphi$;
\item
for every $a\in A_+$ and $\varepsilon>0$ there exists an element $r\in\overline{aAa}$ with $r^2=0$ such that $(a-\varepsilon)_+\in\overline{\linSpan}ArA$.
\item
$\Cu(A)$ is $(k,\omega)$-divisible for each $k\geq 2$: for every $x'\ll x$ in $\Cu(A)$, and every $k\geq 2$, there exists $y\in\Cu(A)$ such that $ky\leq x$ and $x'\leq\infty y$.
\item
$\Cu(A)$ is $(2,\omega)$-divisible.
\end{enumerate}
\end{thm}
\begin{proof}
It is clear that~(1) implies~(2). One can also see that~(2) implies~(3).
It follows from \cite[Theorem~5.3~(i)]{RobRor13Divisibility} that~(4) implies~(1). Further, (4)~is equivalent to~(5); see \autoref{pgr:2OmegaDef}. 
We will now show that~(3) implies~(5).

By \autoref{prp:ScaleDiv}, it suffices to show that every element in $\Sigma_A$ is $(2,\omega)$-divisible.
So let $x\in\Sigma_A$ and $x'\in\Cu(A)$ satisfy $x'\ll x$.
Applying \autoref{prp:ScaleCuA}~(1), we obtain $a\in A_+$ such that $x'\ll[a]\ll x$.
Choose $\varepsilon>0$ such that $x'\ll [(a-\varepsilon )_+]$. 
By assumption, there exists $r\in\overline{aAa}$ such that $r^2=0$ and $(a-\varepsilon )_+\in \overline{\linSpan}ArA$.

Since $r\in\overline{aAa}$, we have $rr^*,r^*r\in \overline{aAa}$. 
Using that these two elements are orthogonal, one gets $rr^*\oplus rr^*\sim rr^*\oplus r^*r\sim rr^* + r^*r\precsim a$.
Thus, the element $c=[rr^*]$ in $\Cu(A)$ satisfies 
\[
2c = [rr^*\oplus rr^*] \leq [a] \ll x, \andSep 
x'\ll [(a-\varepsilon)_+]\leq \infty c,
\]
as required.
\end{proof}

By \cite[Theorem~8.9]{ThiVil21arX:NowhereScattered}, a \ca{} is nowhere scattered if and only if its Cuntz semigroup is weakly $(2,\omega)$-divisible.
Together with \autoref{prp:GlimmFirstChar}, we obtain the following reformulation of the Global Glimm Problem:

\begin{pbm}[Cuntz semigroup reformulation of the Global Glimm Problem]
\label{pbm:ReformulateDiv}
Let~$A$ be a \ca{}.
Assuming that $\Cu(A)$ is weakly $(2,\omega)$-divisible, does it follow that $\Cu(A)$ is $(2,\omega)$-divisible?
\end{pbm}

\begin{rmk}
In \cite{ThiVil22pre:Soft}, we will show that the conditions in \autoref{prp:GlimmFirstChar} are also equivalent to a strengthening of~(6):
for every $x\in\Cu(A)$ and every $k\geq 2$ there exists (a strongly soft element) $y\in\Cu(A)$ such that $ky\leq x\leq\infty y$.
\end{rmk}

Next, we prove permanence properties for $(2,\omega)$-divisibility of \CuSgp{s}, which through our characterization in \autoref{prp:GlimmFirstChar} lead to permanence properties for the Global Glimm Property.

\begin{prp}
\label{prp:DivExtension}
Let $S$ be a \CuSgp{} satisfying \axiomO{5}-\axiomO{8}, and let $I\subseteq S$ be an ideal.
Then $S$ is $(2,\omega)$-divisible if and only if $I$ and $S/I$ are.
\end{prp}
\begin{proof}
Let $\pi\colon S\to S/I$ be the quotient map.
First assume that $S$ is $(2,\omega)$-divisible.
It is clear that $I$ is $(2,\omega)$-divisible as well.
To show that $S/I$ is $(2,\omega)$-divisible, let $\bar{x},\bar{y}\in S/I$ satisfy $\bar{x}\ll\bar{y}$, and let $x,y\in S$ be lifts of $\bar{x}$ and $\bar{y}$ respectively.
Then $\pi(x)\ll\pi(y)$.
Since $\pi$ is a \CuMor{}, we can choose $y'$ such that $\pi(x)\leq\pi(y')$ and $y'\ll y$.
Using that $y$ is $(2,\omega)$-divisible, we obtain $z\in S$ such that $2z\leq y$ and $y'\leq\infty z$.
Thus,
\[
2\pi(z) \leq \pi(y)=\bar{y}, \andSep
\bar{x}=\pi (x) \leq \pi(y') \leq \infty \pi(z),
\]
which shows that $\pi(z)$ has the desired properties.

Conversely, assume that $I$ and $S/I$ are $(2,\omega)$-divisible.
To show that $S$ is $(2,\omega)$-divisible, let $x'',x\in S$ satisfy $x''\ll x$.
We need to find $d\in S$ such that $2d\leq x$ and $x''\leq\infty d$.
Choose $x'\in S$ such that $x''\ll x'\ll x$.

Then $\pi(x')\ll\pi(x)$.
As mentioned in \autoref{pgr:2OmegaDef}, $S/I$ is $(3,\omega)$-divisible. Thus, we obtain $\bar{y}\in S/I$ such that
\[
3\bar{y} \leq \pi(x), \andSep \pi(x') \ll \infty\bar{y}.
\]

Let $y\in S$ be a lift of $\bar{y}$.
Using that $\pi$ is a \CuMor{}, we can choose $y'',y'\in S$ such that
\[
\pi(x')\leq\infty\pi(y''), \andSep
y'' \ll y' \ll y.
\]

Since $3\bar{y} \leq \pi(x)$ and $\pi(x')\leq\infty \pi(y'')$, we obtain $u,v\in I$ such that
\[
3y \leq x+u, \andSep
x' \leq \infty y'' + v.
\]

Set $w:=\infty(u+v)$.
Then
\[
3y' \ll x + w, \andSep
y'' \ll y'.
\]

Applying \cite[Proposition~7.8]{ThiVil21arX:NowhereScattered}, we obtain $z\in S$ such that
\[
3z \ll x, \andSep
y'' \ll z+w.
\]

Choose $z'\in S$ such that $z'\ll z$ and $y'' \ll z'+w$. Applying \autoref{prp:DivO5}, we obtain $c\in S$ such that 
\[
2z'+c \leq x \leq 3c, \andSep
z'\ll c.
\]

Using that $y'' \ll z'+w$ and $v\leq w$, we have
\[
x' 
\leq \infty y'' + v
\leq \infty z' + w.
\]

Let us first indicate the completion of the proof under the additional assumption that~$S$ is countably based.
In this case, for every $s\in S$ and $u=2u\in S$, the infimum $s\wedge u$ exists, and the map $S\to S$, $s\mapsto s\wedge u$, preserves addition, order and suprema of increasing sequences by \cite[Theorems~2.4, 2.5]{AntPerRobThi21Edwards}.
Using also at the first step that $x'\leq \infty z' + w$ and $x'\leq x\leq 3c\leq\infty c$, and using at the last step that $z'\leq c$, we have
\[
x' 
\leq (\infty z' + w)\wedge\infty c
= (\infty z'\wedge\infty c) + (w\wedge \infty c)
= \infty z' + \infty (w\wedge c).
\]

Choose $d$ satisfying
\[
x'' \leq \infty z' + \infty d, \andSep
d \ll w\wedge c.
\]

Using that $w\wedge c$ belongs to $I$ and is therefore $(2,\omega)$-divisible, we obtain $e\in S$ such that
\[
2e \leq w\wedge c, \andSep
d \leq \infty e.
\]

Then the element $z'+e$ has the desired properties:
First, we have
\[
2(z'+e)\leq 2z'+(w\wedge c) \leq 2z'+c\leq x.
\]

Further, we get
\[
x''\leq \infty z' + \infty d \leq \infty z'+\infty e = \infty(z'+e).
\]

One can reduce the general situation to the countably based case using the model theoretic methods developed in \cite[Section~5]{ThiVil21DimCu2}.
Let us outline the argument:
By \cite[Proposition~5.3]{ThiVil21DimCu2}, the properties \axiomO{5}-\axiomO{7} each satisfy the L{\"o}wenheim-Skolem condition.
Using similar arguments, one shows that both \axiomO{8} and $(2,\omega)$-divisibility satisfy the L{\"o}wenheim-Skolem condition as well.
Now, let~$S$ be a \CuSgp{} satisfying \axiomO{5}-\axiomO{8}, and let $I\subseteq S$ be an ideal such that~$I$ and~$S/I$ are $(2,\omega)$-divisible.
Given $x'\ll x$ in $S$, using that \axiomO{5}-\axiomO{8} and $(2,\omega)$-divisibility satisfy the L{\"o}wenheim-Skolem condition, and using the analog of \cite[Lemma~3.2]{Thi20arX:grSubhom} in the setting of \CuSgp{s}, we can find a countably based sub-\CuSgp{} $H\subseteq S$ satisfying \axiomO{5}-\axiomO{8}, containing $x'$ and $x$, and such that $H\cap I$ and $H/(H\cap I)$ are both $(2,\omega)$-divisible.
We obtain that $H$ is $(2,\omega)$-divisible, which allows us to find the element with the desired properties in $H\subseteq S$.

We now present the proof of the general case (avoiding the model theoretic methods):
First, applying \axiomO{6} to
\[
x'' \ll x' \leq \infty z' + w,
\]
we obtain $r\in S$ such that
\[
x'' \ll \infty z' + r, \andSep 
r \leq x',w.
\]

Choose $r'\in S$ with
\[
x'' \leq \infty z' + r', \andSep
r'\ll r.
\]

Next, applying \axiomO{6} for
\[
r' \ll r \leq x' \leq x \leq 3c,
\]
we obtain $s_1,s_2,s_3\in S$ such that
\[
r' \ll s_1+s_2+s_3, \andSep
s_j \leq r,c \quad \text{for $j=1,2,3$}.
\]

Choose $s_j'$ such that
\[
r' \ll s_1'+s_2'+s_3', \andSep
s_j' \leq s_j \quad \text{for $j=1,2,3$}.
\]

Note that $s_j\leq c$ for $j=1,2,3$.
Applying \axiomO{7}, we obtain $s\in S$ such that
\[
s_1',s_2',s_3' \leq s \leq c,s_1+s_2+s_3.
\]

Since $s_1,s_2,s_3\leq w=2w$, we deduce that $s\leq w$.
Thus, $s$ belongs to $I$, and
\[
r' \ll s_1'+s_2'+s_3' \leq 3s.
\]

Choose $s'$ satisfying
\[
s'\ll s, \andSep r' \leq 3s' \leq \infty s'.
\]

Since $I$ is $(2,\omega)$-divisible, we obtain $t\in I\subseteq S$ such that
\[
2t \leq s, \andSep  s'\leq \infty t.
\]

We claim that $z'+t$ has the desired properties.
First, using that $2t\leq s \leq c$, we have
\[
2(z'+t)
\leq 2z'+c
\leq x.
\]

On the other hand, using that $r'\leq\infty s'\leq\infty t$, we get
\[
x'' 
\leq \infty z' + r'
\leq \infty z' + \infty t
= \infty (z'+t),
\]
as desired.
\end{proof}

\begin{thm}
\label{prp:GlimmExtension}
The Global Glimm Property passes to hereditary sub-\ca{s} and quotients.

Further, if $I$ is an ideal in a \ca{} $A$, then $A$ has the Global Glimm Property if and only if $I$ and $A/I$ do.
\end{thm}
\begin{proof}
Let $A$ be a \ca, and let $I\subseteq A$ be an ideal.
Then the inclusion $I\to A$ induces an order-embedding $\Cu(I)\to\Cu(A)$ that identifies $\Cu(I)$ with an ideal in $\Cu(A)$ such that $\Cu(A/I)$ is naturally isomorphic to $\Cu(A)/\Cu(I)$;
see \cite[Proposition~3.3]{CiuRobSan10CuIdealsQuot} and \cite[Section~5.1]{AntPerThi18TensorProdCu}.

It follows from \autoref{prp:DivExtension} that $\Cu(A)$ is $(2,\omega)$-divisible if and only if $\Cu(I)$ and $\Cu(A/I)$ are.
Applying \autoref{prp:GlimmFirstChar}, we obtain that $A$ has the Global Glimm Property if and only if $I$ and $A/I$ do.

Lastly, assume that $A$ has the Global Glimm Property, and let $B\subseteq A$ be a hereditary sub-\ca{}.
Let $I=\overline{\linSpan}ABA$ be the ideal generated by $B$.
Then the inclusion $B\to I$ induces an isomorphism $\Cu(B)\to\Cu(I)$, and arguing as above we deduce that $B$ has the Global Glimm Property.
\end{proof}

\begin{prp}
\label{prp:Limits}
The Global Glimm Property passes to inductive limits.
\end{prp}
\begin{proof}
 Let $((A_\lambda)_{\lambda\in\Lambda},(\varphi_{\mu,\lambda})_{\lambda\leq\mu \text{ in } \Lambda})$ be an inductive system of \ca{s} satisfying the Global Glimm Property, and let $A$ be its inductive limit with limit morphisms $\varphi_\lambda\colon A_\lambda \to A$ for each $\lambda\in\Lambda$.
 
 It follows from \cite[Lemma~3.8]{ThiVil22DimCu} and \cite[Corollary~3.2.9]{AntPerThi18TensorProdCu} that the Cuntz semigroup $\Cu (A)$ together with the \CuMor{s} $\Cu (\varphi_\lambda)$ satisfy the following property: For every pair $x'\ll x$ in $\Cu (A)$, there exist $\lambda\in\Lambda$  and an element $x_\lambda\in \Cu (A_\lambda )$ such that  $x'\ll\Cu (\varphi_\lambda )(x_\lambda)\ll x$.

Thus, given $x\in \Cu (A)$ and $x'\in\Cu (A)$ such that $x'\ll x$, let $\lambda\in\Lambda$ and $x_\lambda\in \Cu (A_\lambda )$ be such that $x'\ll\Cu (\varphi_\lambda )(x_\lambda)\ll x$.

Since $\Cu (\varphi_\lambda )$ is a \CuMor{}, one finds $x_\lambda '\ll x_\lambda$ in $\Cu (A_\lambda )$ such that $x'\ll \Cu (\varphi_\lambda )(x_\lambda ')$. Then, using that $A_\lambda $ has the Global Glimm Property, it follows from \autoref{prp:GlimmFirstChar} that there exists $y_\lambda\in \Cu (A_\lambda )$ satisfying 
\[
2y_\lambda \leq x_\lambda , \andSep x_\lambda '\leq \infty y_\lambda 
\]
in $\Cu (A_\lambda )$.

Therefore, the element $y:=\Cu (\varphi_\lambda )(y_\lambda )$ is such that $2y\leq x$ and $x'\leq \infty y$. This shows that $\Cu (A)$ is $(2,\omega )$-divisible, which by \autoref{prp:GlimmFirstChar} is equivalent to $A$ having the Global Glimm Property.
\end{proof}

\begin{prp}
\label{prp:Morita}
The Global Glimm Property is invariant under Morita equivalence.
In particular, a \ca{} $A$ has Global Glimm Property if and only if its stabilization $A\otimes\KK$ does.
\end{prp}
\begin{proof}
This follows from \autoref{prp:GlimmFirstChar} using that Morita equivalent \ca{s} have isomorphic Cuntz semigroups.
\end{proof}

\section{Ideal-filtered Cuntz semigroups}
\label{sec:IdealFiltered}

In this section, we introduce the notion of ideal-filteredness for \CuSgp{s};
see \autoref{dfn:IdealFiltered}. 
We show in \autoref{prp:DivImplIdealFiltered} that $(2,\omega)$-divisible Cuntz semigroups are always ideal-filtered.
Conversely, we will see in \autoref{prp:MainThmCu} that ideal-filteredness together with property~(V) can be used to characterize when a weakly $(2,\omega)$-divisible \CuSgp{} is $(2,\omega)$-divisible.

We provide a useful characterization of ideal-filteredness (\autoref{prp:CharIdealFilteredO6O7}), which we use to show that a \CuSgp{} is ideal-filtered whenever it satisfies \axiomO{6} and Riesz interpolation (\autoref{prp:InterpolIdealFiltered}), if it satisfies \axiomO{7} and is zero-dimensional (\autoref{prp:Dim0IdealFiltered}), or if it satisfies \axiomO{5}-\axiomO{7} and the tensor product with $\{0,\infty\}$ is algebraic (\autoref{prp:LatAlgebraicIdealFiltered}).

It follows that the Cuntz semigroup of a \ca{} $A$ is ideal-filtered whenever~$A$ has stable rank one (\autoref{prp:IdealFilteredSR1}), real rank zero (\autoref{prp:IdealFilteredRR0}), or if $A$ is separable and has topological dimension zero (\autoref{prp:IdealFilteredTopDimZero}), which includes the case that $A$ has the ideal property.
\autoref{exa:NotIdealFiltered} shows that there exist \ca{s} whose Cuntz semigroups are not ideal-filtered.

Recall that a subset $F$ of a partially ordered set is said to be \emph{filtered} if for every $x,y\in F$ there exists $z\in F$ satisfying $z\leq x,y$.

\begin{dfn}
\label{dfn:IdealFiltered}
We say that a \CuSgp{} $S$ is \emph{ideal-filtered} if for all $v',v,x,y$ in~$S$ satisfying
\[
v' \ll v \ll \infty x, \infty y
\]
there exists $z\in S$ such that
\[
v' \ll \infty z, \andSep z \ll x,y.
\]
\end{dfn}

Recall that an element $x$ in a \CuSgp{} is said to be \emph{compact} if $x\ll x$;
further, $x$ is \emph{full} if it is not contained in a proper ideal.

\begin{rmk}
\label{rmk:IdealFiltered}
Let $S$ be an ideal-filtered \CuSgp, and let $v',v\in S$ satisfy $v'\ll v\leq \infty v'$. For example, one can take $v'=v$ for any compact element $v$.

Consider the set $F:=\{ x \in S : v\leq \infty x\}$ of elements that generate an ideal containing $v$.
Then $F$ is filtered, that is, given $x,y\in F$ there exists $z\in F$ with $z\leq x,y$.

In particular, if $S$ contains a compact full element, then the set of full elements is filtered.
\end{rmk}

\begin{exa}
\label{exa:NotIdealFiltered}
Set $A:=C(S^2)$, the continuous functions on the two-sphere.
Then $\Cu(A)$ is the disjoint union of the submonoid $\Cu(A)_c$ of compact elements and the subsemigroup of noncompact elements.
Further, $\Cu(A)_c$ is isomorphic to the Murray-von Neumann semigroup $V(A)$, which can be identified with a submonoid of $\ZZ\times\ZZ$ as
\[
V(A) \cong \big\{ (0,0) \big\} \cup \big\{ (n,m) : m>0 \big\} \subseteq \ZZ\times\ZZ.
\]
We refer to \cite[Example~6.7]{Thi20RksOps} and \cite[Theorem~1.2]{Rob13CuSpDim2} for details.

Set $x=(0,1)\in V(A)$ and $y=(1,1)\in V(A)$.
Then $x$ and $y$ are compact, full elements in $\Cu(A)$.
However, there exists no full element below $x$ and $y$.
Thus, $\Cu(A)$ is not ideal-filtered.
\end{exa}

It will be useful to consider the following relation:

\begin{ntn}
Let $S$ be a \CuSgp{}.
Given $x,y\in S$, we write $x\llideal y$ if there exists $y'\in S$ such that $x\leq\infty y'$ and $y'\ll y$.
\end{ntn}

The following result summarizes the basic properties of the relation $\llideal$.
We omit the straightforward proofs.

\begin{lma}
\label{prp:FiltRelTauto}
Let $S$ be a \CuSgp{}.
Then the following statements hold:
\begin{enumerate}
\item
If $u\leq\infty x$, $x\llideal y$ and $y\leq\infty z$, then $u\llideal z$.
\item
If $x\llideal y$, then there exists $y'\in S$ such that $x\llideal y'\ll y$.
\end{enumerate}
\end{lma}

\begin{prp}
\label{prp:RephraseIdealFiltered}
A \CuSgp{} $S$ is ideal-filtered if and only if for all $v',v,x,y$ in~$S$ satisfying
\[
v' \ll v \llideal x, y
\]
there exists $z\in S$ such that
\[
v' \llideal z \ll x,y.
\]
\end{prp}
\begin{proof}
Assume first that $S$ is ideal-filtered, and let $v',v,x,y\in S$ be as in the statement.

Then, let $v''\in S$ satisfy $v'\ll v''\ll v$, and note that one has $v'\ll v''\ll \infty x, \infty y$. Thus, applying ideal-filteredness, there exists $z\in S$ such that 
\[
 v'\ll \infty z,\andSep 
 z\ll x,y,
\]
as desired.

Conversely, let $v'\ll v\ll \infty x,\infty y$ in $S$ and assume that the stated property is satisfied. As above, take $v''\in S$ with $v'\ll v''\ll v$.

By assumption, we find $z\in S$ such that
\[
v'' \llideal z \ll x,y
\]
and, since $v'\ll v''\leq \infty z$, it follows that $S$ is ideal-filtered.
\end{proof}

\begin{prp}
\label{prp:CharIdealFiltered}
Let $S$ be a \CuSgp{}.
Then $S$ is ideal-filtered if and only if~$S$ satisfies the following statements:
\begin{enumerate}
\item
For all $v',v,x$ in~$S$ satisfying
\[
v' \ll v \llideal x
\]
there exists $z\in S$ such that
\[
v' \llideal z \llideal v, \andSep z\ll x.
\]
\item
For all $x',x,y$ in~$S$ satisfying
\[
x'\ll x\ll 2y
\]
there exists $z\in S$ such that
\[
x' \llideal z\ll x,y.
\]
\end{enumerate}
\end{prp}
\begin{proof}
We first assume that $S$ is ideal-filtered.
To verify~(1), let $v',v,x$ in~$S$ satisfy $v' \ll v \llideal x$.
Choose $w\in S$ such that $v'\ll w\ll v$.
Since $\ll$ is stronger than $\llideal$, we have
\[
v'\ll w \llideal x,v.
\]

Applying \autoref{prp:RephraseIdealFiltered}, we obtain $z\in S$ such that
\[
v'\llideal z\ll x,v,
\]
which shows that $z$ has the desired properties.

To verify~(2), let $x',x,y$ in~$S$ satisfy $x'\ll x\ll 2y$.
Choose $w\in S$ such that $x'\ll w\ll x$.
Then
\[
x'\ll w \llideal x,y.
\]

Applying \autoref{prp:RephraseIdealFiltered} again, we obtain $z\in S$ such that
\[
x'\llideal z\ll x,y.
\]

Conversely, assume that $S$ satisfies~(1) and~(2).
We first prove the following strengthening of~(2):

\textbf{Claim: }\emph{For all $x',x,y\in S$ and $n\geq 1$ satisfying
\[
x'\ll x\ll 2^n y
\]
there exists $z\in S$ such that
\[
x' \llideal z\ll x,y.
\]}

We prove the claim by induction over $n$.
The case $n=1$ holds by assumption. Thus, assume that the claim holds for some $n$, and let $x',x,y\in S$ satisfy $x'\ll x\ll 2^{n+1} y$.
Then
\[
x'\ll x\ll 2^{n+1} y = 2(2^n y),
\]
 and applying~(2) we obtain $z_0\in S$ such that
\[
x'\llideal z_0 \ll x,2^ny.
\]

Using \autoref{prp:FiltRelTauto} (2), choose $z_0'\in S$ such that $x'\llideal z_0'\ll z_0$.
By the induction assumption applied at $z_0'\ll z_0\ll 2^ny$, we obtain $z\in S$ such that
\[
z_0'\llideal z \ll z_0,y
\]
and, consequently,
\[
x' \llideal z_0' \llideal z, \quad
z\ll z_0\ll x, \andSep
z\ll y,
\]
as desired.
This proves the claim.

To verify that $S$ is ideal-filtered, let $v',v,x,y\in S$ satisfy $v' \ll v \llideal x, y$.
Applying~(1) to $v'\ll v\llideal x$, we obtain $w\in S$ such that
\[
v' \llideal w \llideal v, \andSep w\ll x.
\]

Choose $w'',w'\in S$ such that
\[
v'\llideal w''\ll w'\ll w.
\]

We have $w'\ll w\llideal v\llideal y$, and so $w'\ll\infty y$, which allows us to choose $n\geq 1$ such that $w'\ll 2^ny$.
Applying the Claim to $w''\ll w'\ll 2^ny$, we obtain $z\in S$ such that
\[
w'' \llideal z\ll w',y,
\]
which implies
\[
v' \llideal w''\llideal z, \quad
z \ll w' \ll w\ll x, \andSep
z \ll y,
\]
as desired.
\end{proof}

Next, we show that statement~(1) in \autoref{prp:CharIdealFiltered} is always satisfied for \CuSgp{s} satisfying \axiomO{6} and \axiomO{7}.

\begin{lma}
\label{prp:Cond1Automatic}
Let $S$ be a \CuSgp{} satisfying \axiomO{6} and \axiomO{7}, and let $v',v,x\in S$ satisfy
\[
v' \ll v \leq \infty x.
\]

Then there exists $z\in S$ such that
\[
v' \llideal z \llideal v, \andSep z\ll x.
\]
\end{lma}
\begin{proof}
Choose $v''\in S$ such that $v'\ll v''\ll v$.
Since $v''\ll v\leq\infty x$, we can find $n\geq 1$ such that $v''\ll nx$.
Applying \axiomO{6} for $v'\ll v''\ll nx$, we obtain elements $z_1,\ldots,z_n\in S$ such that
\[
v'\ll z_1+\ldots+z_n, \quad
z_1\ll v'',x, \quad\ldots,\quad
z_n\ll v'',x.
\]

Choose $z_1',\ldots,z_n'\in S$ such that
\[
v'\ll z_1'+\ldots+z_n', \quad
z_1'\ll z_1, \quad \ldots, \quad
z_n'\ll z_n.
\]

Applying \autoref{prp:RefO7} for $z_j'\ll z_j\leq x$, we obtain $z\in S$ such that
\[
z_1',\ldots,z_n'\ll z \ll x,z_1+\ldots+z_n.
\]

This implies that
\[
v' \ll z_1'+\ldots+z_n'\leq \infty z, \quad
z \ll z_1+\ldots+z_n \leq \infty v'' \leq \infty v, \andSep
z\ll x, 
\]
as required.
\end{proof}

\begin{prp}
\label{prp:CharIdealFilteredO6O7}
Let $S$ be a \CuSgp{} satisfying \axiomO{6} and \axiomO{7}.
Then $S$ is ideal-filtered if and only if for all $x',x,y$ in~$S$ satisfying
\[
x'\ll x\ll 2y
\]
there exists $z\in S$ such that
\[
x' \llideal z\ll x,y.
\]
\end{prp}
\begin{proof}
This follows from \autoref{prp:CharIdealFiltered} and \autoref{prp:Cond1Automatic}.
\end{proof}

\begin{prp}
\label{prp:DivImplIdealFiltered}
Let $S$ be a $(2,\omega)$-divisible \CuSgp{} satisfying \axiomO{6} and \axiomO{7}.
Then $S$ is ideal-filtered.
\end{prp}
\begin{proof}
It suffices to verify the statement in \autoref{prp:CharIdealFilteredO6O7}. Thus, let $x',x,y$ in~$S$ satisfy $x'\ll x\ll 2y$.
Using that $x$ is $(2,\omega)$-divisible, we obtain $d\in S$ such that $2d\ll x$ and $x'\llideal d$.
By \autoref{prp:FiltRelTauto}, there exists $d'\in S$ satisfying
\[
x'\llideal d'\ll d.
\]

Applying \axiomO{6} to $d'\ll d \ll 2y$, we obtain $z_1,z_2\in S$ such that
\[
d' \ll z_1+z_2, \andSep
z_1,z_2\ll d,y.
\]

Choose $z_1',z_2'\in S$ such that
\[
d'\ll z_1'+z_2', \quad
z_1'\ll z_1, \andSep
z_2'\ll z_2.
\]

Applying \axiomO{7} to $z_1'\ll z_1 \ll y$ and $z_2'\ll z_2 \ll y$, we obtain $z\in S$ satisfying
\[
z_1',z_2' \ll z \ll y,z_1+z_2,
\]
which implies
\[
x' \llideal d' \ll z_1'+z_2' \llideal z, \quad
z \ll z_1+z_2 \leq d+d \ll x, \andSep
z \ll y.
\]

This shows that $z$ has the desired properties.
\end{proof}

\begin{prp}
\label{prp:IdeFilStabCas}
Let $A$ be a stable \ca{}. 
Then, $\Cu(A)$ is ideal-filtered if and only if for every $a\in A_+$, $b\in (\overline{aAa}\otimes M_2)_+$ and $\varepsilon>0$, there exists $c\in\overline{aAa}$ such that $c\precsim b$ and such that $(b-\varepsilon)_+\in\overline{\linSpan}AcA$.
\end{prp}
\begin{proof}
Since the Cuntz semigroup of every \ca{} satisfies \axiomO{6} and \axiomO{7}, we know by \autoref{prp:CharIdealFilteredO6O7} that $\Cu(A)$ is ideal-filtered if and only if for all $x',x,y$ in~$\Cu (A)$ satisfying $x'\ll x\ll 2y$ there exists $z\in \Cu (A)$ such that
\[
x' \llideal z\leq x,y.
\] 

Thus, to prove the backward implication, let $x'\ll x\ll 2y$ and take $a,b\in A_+$ such that $x=[b]$ and $y=[a]$. 
Let $\varepsilon>0$ be such that $x'\ll [(b-2\varepsilon )_+]$.
Using that~$b$ is Cuntz subequivalent to $a\otimes 1$ (a strictly positive element in $\overline{aAa}\otimes M_2$), we can choose $r\in A\otimes M_2$ such that
\[
(b-\varepsilon )_+ = rr^*, \andSep
r^*r\in \overline{aAa}\otimes M_2;
\]
see, for example, \cite[Theorem~2.30]{Thi17:CuLectureNotes}.

By assumption, there exists $c\in(\overline{aAa})_+$ with $c\precsim r^*r$ and such that $(r^*r-\varepsilon )_+$ belongs in the ideal generated by $c$.
We claim that $[c]$ has the required properties.
Indeed, we have
\[
[c] \leq [a] = y, \andSep
[c] \leq [r^*r] = [rr^*] \leq [b] = x.
\]

Further, using that $(rr^*-\varepsilon )_+\sim (r^*r-\varepsilon )_+$ at the third step (see, for example, \cite[Corollary~2.53]{Thi17:CuLectureNotes}), we get
\[
x' \ll [(b-2\varepsilon)_+] = [(rr^*-\varepsilon )_+] = [(r^*r-\varepsilon )_+] \leq \infty[c],
\]
as required.

To prove the forward implication, let $a,b,\varepsilon$ be as in the statement. 
Then, since $b\in \overline{aAa}\otimes M_2$ and $[(b-\varepsilon /2)_+]\ll [b]$ in $\Cu (A)$, we have
\[
[(b-\varepsilon )_+]\ll [(b-\varepsilon/2)_+ ]\ll 2[a].
\]

Using that $\Cu (A)$ is ideal-filtered, we obtain $c\in A_+$ such that
\[
[(b-\varepsilon )_+] \llideal [c] \ll  [(b-\varepsilon/2)_+ ], [a].
\]

Choose $\delta>0$ in such a way that
\[
[(b-\varepsilon )_+] \leq \infty[(c-\delta)_+].
\]

Since $c\precsim a$, we obtain $r$ such that $(c-\delta)_+=rr^*$ and $r^*r\in\overline{aAa}$;
see \cite[Theorem~2.30]{Thi17:CuLectureNotes}.
Then $r^*r$ has the desired properties.
\end{proof}

\begin{lma}
\label{prp:InterpolIdealFiltered}
Let $S$ be a \CuSgp{} satisfying \axiomO{6} and Riesz interpolation.
Then $S$ is ideal-filtered.
\end{lma}
\begin{proof}
It is readily checked that the Riesz interpolation property implies \axiomO{7}.
Thus, it suffices to verify the statement in \autoref{prp:CharIdealFilteredO6O7}.
Let $x',x,y$ in~$S$ satisfy $x'\ll x\ll 2y$.
Applying \axiomO{6}, we obtain $z_1,z_2\in S$ such that
\[
x'\ll z_1+z_2, \quad
z_1\ll x,y, \andSep
z_2\ll x,y.
\]

Now choose $\tilde{x},\tilde{y}\in S$ such that
\[
z_1\ll \tilde{x},\tilde{y}, \quad
z_2\ll \tilde{x},\tilde{y}, \quad
\tilde{x}\ll x, \andSep
\tilde{y}\ll y.
\]

Using Riesz interpolation for $z_1,z_2\leq \tilde{x},\tilde{y}$, we obtain $z\in S$ such that
\[
z_1,z_2\leq z\leq \tilde{x},\tilde{y}
\]
and, therefore, we have
\[
x' \ll z_1+z_2 \leq \infty z, \andSep z\ll x,y,
\]
which shows that $z$ has the desired properties.
\end{proof}

\begin{thm}
\label{prp:IdealFilteredSR1}
Let $A$ be a \ca{} of stable rank one.
Then $\Cu(A)$ is ideal-filtered.
\end{thm}
\begin{proof}
The Cuntz semigroup of every \ca{} satisfies \axiomO{6};
see \cite[Proposition~5.1.1]{Rob13Cone}.
By \cite[Theorem~3.5]{AntPerRobThi22CuntzSR1}, the Cuntz semigroup of every \ca{} of stable rank one satisfies Riesz interpolation.
Thus, we may apply \autoref{prp:InterpolIdealFiltered} to deduce that $\Cu(A)$ is ideal-filtered.
\end{proof}

Recall from \cite[Definition~3.1]{ThiVil22DimCu} that a \CuSgp{} $S$ is said to have \emph{dimension at most $n$}, in symbols $\dim (S)\leq n$, if whenever $x'\ll x\ll y_1+\ldots +y_r$ in~$S$ there exist elements $z_{j,k}\in S$ for $j = 1,\ldots ,r$ and $k = 0,\ldots,n$ such that
\begin{enumerate}
\item[(i)]
$z_{j,k}\ll y_j$ for each $j$ and $k$;
\item[(ii)]
$x'\ll \sum_{j,k} z_{j,k}$;
\item[(iii)]
$\sum_j z_{j,k}\ll x$ for each $k$.
\end{enumerate}

\begin{lma}
\label{prp:Dim0IdealFiltered}
Let $S$ be a \CuSgp{} satisfying \axiomO{7} and $\dim(S)=0$.
Then $S$ is ideal-filtered.
\end{lma}
\begin{proof}
As noted in \cite[Remark~5.2]{ThiVil22DimCu}, $\dim(S)=0$ implies \axiomO{6}.
Thus, it suffices to verify the statement in \autoref{prp:CharIdealFilteredO6O7}.
Let $x',x,y$ in~$S$ satisfy $x'\ll x\ll 2y$.
Applying that $S$ is zero-dimensional, we obtain $z_1,z_2\in S$ satisfying
\[
z_1,z_2\ll y, \andSep
x' \ll z_1+z_2 \ll x.
\]

Choose $z_1',z_2'\in S$ such that
\[
z_1'\ll z_1, \quad
z_2'\ll z_2, \andSep
x' \ll z_1'+z_2'.
\]

Applying \axiomO{7} to $z_1'\ll z_1\ll y$ and $z_2'\ll z_2\ll y$, we obtain $z\in S$ such that
\[
z_1',z_2'\ll z \ll y, z_1+z_2
\]
and, consequently,
\[
x' \ll z_1'+z_2' \leq \infty z, \quad
z \ll z_1+z_2 \ll x, \andSep
z \ll y,
\]
which shows that $z$ has the desired properties.
\end{proof}

In \autoref{prp:IdealFilteredTopDimZero} we will generalize \autoref{prp:IdealFilteredRR0} below to the case of \ca{s} with topological dimension zero in the separable case.

\begin{thm}
\label{prp:IdealFilteredRR0}
Let $A$ be a \ca{} of real rank zero.
Then $\Cu(A)$ is ideal-filtered.
\end{thm}
\begin{proof}
The Cuntz semigroup of every \ca{} satisfies \axiomO{7};
see \cite[Proposition~2.2]{AntPerRobThi21Edwards}.
By \cite[Proposition~5.4]{ThiVil22DimCu}, the Cuntz semigroup of every \ca{} of real rank zero is zero-dimensional.
Thus, we may apply \autoref{prp:Dim0IdealFiltered} to deduce that $\Cu(A)$ is ideal-filtered.
\end{proof}

\begin{lma}
\label{prp:CharLatAlgbraic}
Let $S$ be a \CuSgp{} satisfying \axiomO{6} and \axiomO{7}.
Then the following are equivalent:
\begin{enumerate}
\item
for all $x',x\in S$ satisfying $x'\ll x$ there exist $y',y\in S$ such that $x'\ll y'\ll y\ll x$ and $y\ll \infty y'$;
\item
$S\otimes\{0,\infty\}$ is algebraic.
\end{enumerate}
\end{lma}
\begin{proof}
We identify $S\otimes\{0,\infty\}$ with the set $\Lat_f(S)$ of singly generated ideals in $S$;
see \cite[Proposition~7.2.3]{AntPerThi18TensorProdCu}.
Given $x\in S$, we set $\langle x\rangle:=\{y\in S:y\leq\infty x\}$, the ideal generated by $x$.
By \cite[Proposition~5.1.7]{AntPerThi18TensorProdCu}, the map $S\to\Lat_f(S)$, $x\mapsto\langle x\rangle$, is a surjective \CuMor.
Thus, given $x,y\in S$, we have $\langle x \rangle \ll \langle y \rangle$ in $\Lat_f(S)$ if and only if there exists $y'\in S$ such that $x\leq\infty y'$ and $y'\ll y$.

To see that~(1) implies~(2), let $I',I \in \Lat_f(S)$ satisfy $I'\ll I$.
We need to find $J\in\Lat_f(S)$ satisfying $I'\leq J\ll J \leq I$. Choose $x\in S$ such that $I=\langle x \rangle$.
Then, we obtain $x'$ such that $I'\leq\langle x' \rangle$ and $x'\ll x$.
Applying the assumption, we get $y',y$ such that
\[
x'\ll y'\ll y\ll x, \andSep
y\ll \infty y'.
\]

Set $J:=\langle y \rangle$.
Using that $y\ll \infty y'$ at the first step, one has
\[
J \leq \langle y' \rangle \ll \langle y \rangle = J,
\]
showing that $J$ has the desired properties.

Conversely, assume that $\Lat_f(S)$ is algebraic. 
To verify~(1), let $x'',x\in S$ be such that $x''\ll x$.
Choose $x'$ satisfying $x''\ll x'\ll x$.
Then, since $x\mapsto \langle x\rangle$ is a \CuMor{}, we get $\langle x' \rangle \ll \langle x \rangle$.
We obtain $J\in\Lat_f(S)$ such that $\langle x' \rangle \leq J \ll J \leq \langle x \rangle$.
Thus, we obtain $z'',z$ such that
\[
x' \leq \infty z'', \andSep 
z''\ll z \leq \infty x, \infty z''.
\]

Choose $z'$ such that
\[
z'' \ll z' \ll z.
\]

We obtain $n$ satisfying $z' \leq n x$.
Applying \axiomO{6} for $z''\ll z'\leq nx$, there exist $e_1,\ldots,e_n$ such that
\[
z'' \ll e_1+\ldots+e_n, \andSep
e_j \ll z',x \text{ for } j=1,\ldots,n.
\]

Choose $e_1',\ldots,e_n'\in S$ such that
\[
z''\ll e_1'+\ldots+e_n', \andSep
e_j'\ll e_j \text{ for } j=1,\ldots,n.
\]

Applying \axiomO{7} for $x''\ll x'\leq x$ and $e_j'\ll e_j \ll x$ for $j=1,\ldots,n$, we obtain $y$ such that
\[
x'',e_1',\ldots,e_n'\ll y \ll x, x'+e_1+\ldots+e_n.
\]

Choose $y'$ such that
\[
x'',e_1',\ldots,e_n'\ll y' \ll y.
\]

Then $x''\ll y'\ll y\ll x$ and, since $z'\ll z\leq \infty z''$, we also get 
\[
y \ll x'+e_1+\ldots+e_n \leq \infty z' \leq \infty z'' \leq \infty(e_1'+\ldots+e_n') \leq \infty y',
\]
as desired.
\end{proof}

\begin{lma}
\label{prp:LatAlgebraicIdealFiltered}
Let $S$ be a \CuSgp{} satisfying \axiomO{5}-\axiomO{7}.
Assume that $S\otimes\{0,\infty\}$ is algebraic.
Then $S$ is ideal-filtered.
\end{lma}
\begin{proof}
In order to apply \autoref{prp:CharIdealFilteredO6O7}, let $x',x,y\in S$ satisfy $x'\ll x\ll 2y$.
We need to find $z$ such that $x'\llideal z\ll x,y$.
First, applying \axiomO{6}, we obtain $e,f\in S$ such that
\[
x' \ll e+f, \quad
e\leq x,y, \andSep
f\leq x,y.
\]

Choose $e'',e',f'\in S$ satisfying
\[
x' \ll e''+f', \quad
e'' \ll e' \ll e, \andSep
f' \ll f.
\]

Using \autoref{prp:CharLatAlgbraic}, we may assume that $e\leq\infty e''$.
Applying \axiomO{5} for $e'\ll e\leq x$, we obtain $c\in S$ such that
\[
e'+c \leq x \leq e+c
\]
and, consequently,
\[
f' \ll f \leq x \leq e+c.
\]

Applying \axiomO{6}, we obtain $d\in S$ such that
\[
f' \ll e+d, \andSep
d\leq f,c.
\]

Choose $d'\in S$ such that
\[
f' \ll e+d', \andSep
d' \ll d.
\]

Then
\[
e' \leq e \leq y, \andSep
d \leq f \leq y.
\]

Applying \axiomO{7} for $e''\ll e'\leq y$ and $d'\ll d\leq y$, we obtain $z\in S$ such that
\[
e'',d'\ll z, \andSep
z \ll y, e'+d.
\]

Using that $d\leq c$, it follows that
\[
z\ll e'+d \leq e'+c \leq x.
\]
Further, since $e\leq\infty e''$, we have
\[
x' 
\ll e''+f'
\leq e''+e+d'
\leq \infty(e''+d')
\leq \infty z,
\]
which shows that $z$ has the desired properties.
\end{proof}

A \ca{} $A$ is said to have \emph{topological dimension zero} if its primitive ideal space $\Prim(A)$ has a basis consisting of compact-open subsets;
see \cite[Remark~2.5(vi)]{BroPed09Limits}.

\begin{prp}
\label{prp:CharTopDimZero}
Let $A$ be a separable \ca.
Then the following are equivalent:
\begin{enumerate}
\item
$A$ has topological dimension zero;
\item
$A\otimes\mathcal{O}_2$ has real rank zero;
\item
the \CuSgp{} $\Cu(A)\otimes\{0,\infty\}$ is algebraic.
\end{enumerate}
\end{prp}
\begin{proof}
The equivalence of~(1) and~(2) was shown in \cite[Corollary~4.3]{PasRor07PIRR0}.

If $A\otimes\mathcal{O}_2$ has real rank zero, then $\Cu(A\otimes\mathcal{O}_2)$ is algebraic; see \cite{CowEllIva08CuInv}.
By \cite[Corollary~7.2.15]{AntPerThi18TensorProdCu}, we have $\Cu(A\otimes\mathcal{O}_2)\cong\Cu(A)\otimes\{0,\infty\}$, which shows that~(2) implies~(3).

To show that~(3) implies~(1), we may assume that $A$ is stable.  
Then, for every $b\in A_+$ such that $y'\ll [b]\ll \infty y'$ for some $y'$ in $\Cu (A)$, let $J_b$ be the ideal generated by $b$ in $A$. 
Denote by $\mathcal{B}$ the collection of all such ideals $J_b$.

Note that, by \cite[Proposition~3.1]{TikTom15CuSgpNonunital}, $\Prim(J_b)$ is compact for every $J_b\in \mathcal{B}$. 
Identify $\Prim(J_b)$ with an open subset in $\Prim (A)$; 
see, for example, \cite[II.6.5.4]{Bla06OpAlgs}. 
We claim that the family
\[
\big\{ \Prim(J_b) : J_b\in\mathcal{B} \big\}
\]
is a basis of $\Prim(A)$.

To see this, it is enough to check that every singly generated ideal of $A$ can be written as a union of ideals in $\mathcal{B}$, since it is readily checked that the primitive ideal spaces of such ideals form a basis in $\Prim(A)$.

Thus, let $I\subseteq A$ be an ideal generated by some $a\in A_+$, and set $x:=[a]\in \Cu(A)$.
For every $\varepsilon >0$, consider the element $(a-\varepsilon )_+$ and note that $\Prim (I)$ is the increasing union of the primitive ideal spaces of the ideals generated by each of these elements. Consequently, we only need to show that for every $\varepsilon >0$ there exists $J_b\in\mathcal{B}$ with $(a-\varepsilon )_+\in J_b$ and $J_b\subseteq I$.

Take $\varepsilon >0$, and let $x'$ be the class of $(a-\varepsilon )_+$ in $\Cu (A)$. 
Using (3) and \autoref{prp:CharLatAlgbraic}, we obtain $y',y\in\Cu (A)$ satisfying
\[
 x'\ll y'\ll y\ll x,\andSep y\ll \infty y'.
\]

Let $b\in A_+$ be such that $y=[b]$. 
Then, $J_b\in\mathcal{B}$ and, since $x'=[(a-\varepsilon )_+]\ll [b]$, it follows that $(a-\varepsilon )_+\in J_b$. Further, since we also have $[b]\ll x=[a]$, we also get $J_b\subseteq I$, as desired.
\end{proof}

A \ca{} $A$ is said to have the \emph{ideal property} if for every pair of ideals $I\subseteq J\subseteq A$ with $I\neq J$, there exists a projection $p\in J\setminus I$.
This implies that $A$ has topological dimension zero;
see \cite[Theorem~2.8]{PasPhi17WeakIPTopDimZero}.

\begin{thm}
\label{prp:IdealFilteredTopDimZero}
Let $A$ be a separable \ca{} that has topological dimension zero.
(For example, $A$ has the ideal property.)
Then $\Cu(A)$ is ideal-filtered.
\end{thm}
\begin{proof}
The Cuntz semigroup of every \ca{} satisfies \axiomO{5}-\axiomO{7}.
Further, it follows from \autoref{prp:CharTopDimZero} that $\Cu(A)\otimes\{0,\infty\}$ is algebraic.
Thus, we may apply \autoref{prp:LatAlgebraicIdealFiltered} to deduce that $\Cu(A)$ is ideal-filtered.
\end{proof}

\begin{pbm}
Which Cuntz semigroups of \ca{s} are ideal-filtered?
\end{pbm}

\section{Property~(V)}
\label{sec:V}

In this section, we introduce property~(V) for \CuSgp{s};
see \autoref{dfn:PrpV}. 
We show in \autoref{prp:DivImplV} that $(2,\omega)$-divisible \CuSgp{s} have property~(V).
Conversely, we will see in \autoref{prp:MainThmCu} that property~(V) together with ideal-filteredness can be used to characterize when a weakly $(2,\omega)$-divisible \CuSgp{} is $(2,\omega)$-divisible.

We show that the Cuntz semigroup of a \ca{} $A$ has property~(V) whenever~$A$ is residually stably finite (\autoref{prp:V-RSF}), in particular if $A$ has stable rank one, if $A$ has real rank zero (\autoref{prp:V-rr0}), and if $A$ is separable and has topological dimension zero and strict comparison of positive elements (\autoref{prp:V-AlmUnp}).
We do not know of any \ca{} whose Cuntz semigroup does not have property~(V);
see \autoref{qst:V}.

\begin{dfn}
\label{dfn:PrpV}
A \CuSgp{} $S$ is said to have \emph{property (V)} if for all elements $c,d_j',d_j,x\in S$ satisfying $d_j'\ll d_j$ for $j=1,2$, and 
\[
d_1,d_2 \ll c, \andSep
c+d_1,c+d_2\ll x,
\] 
there exist $y,z\in S$ such that 
\[
y+z\leq x, \andSep 
d_1'+d_2' \leq \infty y, \infty z.
\]
\end{dfn}

\begin{rmk}
\label{rmk:PrpV}
A \CuSgp{} $S$ is said to be \emph{sup-semilattice ordered} if suprema exist and we have $x+(y\vee z)=(x+y)\vee(x+z)$ for every $x,y,z\in S$; see, for example, \cite{Vil21arX:CommCuAI}.

If $S$ is sup-semilattice ordered, then it satisfies property~(V) with $y=c$ and $z=d_1\vee d_2$.
One can therefore consider property~(V) as a (very weak) form of being sup-semilattice ordered.
\end{rmk}

\begin{prp}
\label{prp:DivImplV}
Every $(2,\omega)$-divisible \CuSgp{} has property (V).
\end{prp}
\begin{proof}
Let $S$ be a $(2,\omega)$-divisible \CuSgp{}.
To verify that $S$ has property~(V), let $c,d_j',d_j,x\in S$ satisfy $d_j'\ll d_j$ for $j=1,2$, and 
\[
d_1,d_2 \ll c, \andSep
c+d_1,c+d_2\ll x.
\] 

Now, take $x'$ with $c+d_1,c+d_2\leq x' \ll x$.
Using that $S$ is $(2,\omega)$-divisible, we obtain $u\in S$ such that
\[
2u \leq x, \andSep 
x'\leq \infty u.
\]

Then, $y:=u$ and $z:=u$ have the desired properties.
\end{proof}

We show below that property~(V) holds for Cuntz semigroups of large classes of \ca{s}.
It is possible that property~(V) always holds:

\begin{qst}
\label{qst:V}
Does the Cuntz semigroup of every \ca{} have property~(V)?
\end{qst}

A \CuSgp{} $S$ is said to be \emph{stably finite} if for all $x,y\in S$ with $x+y\ll x$ we have $y=0$. Note that this is more restrictive than the definition in \cite{AntPerThi18TensorProdCu}, although both notions agree if $S$ is simple.

Further, we will say that $S$ is \emph{residually stably finite} if every quotient of $S$ is stably finite.

A stably finite \ca{} has a stably finite Cuntz semigroup, and a residually stably finite \ca{} has a residually stably finite Cuntz semigroup.

\begin{lma}
\label{prp:V-RSF-pre}
Let $S$ be a residually stably finite \CuSgp{} satisfying \axiomO{5}.
Then~$S$ has property~(V).
\end{lma}
\begin{proof}
To verify property (V), let $c,d_j',d_j,x\in S$ satisfy $d_j'\ll d_j$ for $j=1,2$, and 
\[
d_1,d_2 \ll c, \andSep
c+d_1,c+d_2\ll x.
\] 

Choose $c'',c'\in S$ such that
\[
d_1,d_2 \ll c'' \ll c' \ll c.
\]

Applying \axiomO{5} for $c'' \ll c'\leq x$, we find $y$ such that 
\[
c''+y \leq x \leq c'+y.
\]

Given $j\in\{1,2\}$, we have $c+d_j \leq x \leq c'+y$.
In the quotient $S/\langle y \rangle$ one has 
\[
[c]+[d_j] \leq [c'] \ll [c].
\]

Since $S/\langle y \rangle$ is stably finite, we deduce that $[d_j]=0$, and thus $d_j\leq \infty y$.
It follows that $y$ and $z:=c''$ have the desired properties to verify property~(V).
\end{proof}

\begin{thm}
\label{prp:V-RSF}
Let $A$ be a residually stably finite \ca{}.
Then $\Cu(A)$ has property~(V).
\end{thm}
\begin{proof}
The Cuntz semigroup of every \ca{} satisfies \axiomO{5};
see \cite[Proposition~4-6]{AntPerThi18TensorProdCu}.
Since $A$ is residually stably finite, so is $\Cu(A)$.
Now the result follows from \autoref{prp:V-RSF-pre}.
\end{proof}

\begin{lma}
\label{prp:RefinementV}
Let $M$ be a refinement monoid, equipped with the algebraic pre-order~$\leq$.
Let $a,b_1,b_2,c\in M$ satisfy
\[
b_1 ,b_2\leq a, \andSep
a+b_1,a+b_2\leq c.
\]

Then there exist $e,f\in M$ such that
\[
e+f\leq c, \quad
b_1,b_2\leq e, \andSep
b_1,b_2\leq 2f.
\]
\end{lma}
\begin{proof}
Choose $r,s\in M$ such that
\[
r+a+b_1 = s+a+b_2 = c.
\]

Applying refinement for $s+a+b_2 = (r+a) + b_1$, we obtain a refinement matrix:
\[
\begin{array}{|c|c|c|}
\cline{2-3}
\multicolumn{1}{l|}{} & r+a & b_1 \\
\hline
s & z_{1,1} & z_{1,2} \\
\hline
a & z_{2,1} & z_{2,2} \\
\hline
b_2 & z_{3,1} & z_{3,2} \\
\hline
\end{array}
\]

Using that $r+a\geq a\geq b_1$, we may apply \cite[Lemma~3.1]{OrtPerRor11CoronaRefinement} to guarantee that $z_{2,1}\geq z_{2,2}$.
(The assumption that the refinement monoid is conical is not used in the proof of \cite[Lemma~3.1]{OrtPerRor11CoronaRefinement}.)
Now set
\[
e := z_{1,2} + z_{2,2} + z_{3,2} + z_{3,1}, \andSep
f:= z_{2,1}.
\]

We clearly have $e+f\leq c$ and $b_1,b_2\leq e$.
Using that $z_{2,1}\geq z_{2,2}$, we have
\[
b_1,b_2 \leq a = z_{2,1}+z_{2,2} \leq 2 z_{2,1} = 2f,
\]
which shows that $e$ and $f$ have the claimed properties.
\end{proof}

\begin{thm}
\label{prp:V-rr0}
Cuntz semigroups of \ca{s} of real rank zero have property~(V).
\end{thm}
\begin{proof}
Let $A$ be a \ca{} of real rank zero.
Let $C\subseteq\Cu(A)$ denote the monoid of Cuntz classes of projections in $A\otimes\KK$.
Since $A$ has real rank zero, $C$ is sup-dense in $\Cu(A)$.
To verify that $\Cu(A)$ has property (V), let $c,d_j',d_j,x\in\Cu(A)$ satisfy $d_j'\ll d_j$ for $j=1,2$, and 
\[
d_1,d_2 \ll c, \andSep
c+d_1,c+d_2\ll x.
\]

Choose $c',x'\in\Cu(A)$ such that
\[
d_1,d_2 \leq c' \ll c, \andSep
c+d_1,c+d_2 \leq x' \ll x.
\]

Using that $C$ is sup-dense, we obtain projections $p,q_1,q_2,r\in A\otimes\KK$ such that
\[
c' \leq [p] \leq c, \quad
d_1' \leq [q_1] \leq d_1, \quad
d_2' \leq [q_2] \leq d_2, \andSep
x' \leq [r] \leq x.
\]

Abusing notation, we also use $[g]$ to denote the equivalence class of a projection $g\in A\otimes\KK$ in the Murray-von Neumann semigroup $V(A)$.
Among projections, Cuntz subequivalence agrees with Murray-von Neumann subequivalence (but Cuntz equivalence need not imply Murray-von Neumann equivalence).
Thus, in $V(A)$ we have
\[
[q_1],[q_2] \leq [p], \andSep
[p]+[q_1], [p]+[q_1] \leq [r].
\]

By \cite[Lemma~2.3]{AraPar96RefMonWkComparability}, $V(A)$ is a refinement monoid.
Applying \autoref{prp:RefinementV}, we obtain projections $e,f\in A\otimes\KK$ such that
\[
[e]+[f] \leq [r], \quad
[q_1],[q_2] \leq [e], \andSep
[q_1],[q_2] \leq 2[f].
\]
Then the Cuntz classes $y:=[e]$ and $z:=[f]$ have the desired properties.
\end{proof}

\begin{rmk}
If $A$ is a \ca{} of real rank zero, then its Murray-von Neumann semigroup $V(A)$ is a refinement monoid by \cite[Lemma~2.3]{AraPar96RefMonWkComparability}.
One is tempted to conclude that the monoid $\Cu(A)_c$ of compact elements in $\Cu(A)$ is a refinement monoid as well.
However, this is not clear, since $\Cu(A)_c$ is `only' isomorphic to the maximal antisymmetric quotient of $V(A)$, and there exist examples of refinement monoids whose maximal antisymmetric quotient is not a refinement monoid, \cite{Mor02RefinementMaxAntisymQuot}.
On the other hand, we do not know if the example from \cite{Mor02RefinementMaxAntisymQuot} can be realized as the Murray-von Neumann semigroup of a \ca.
\end{rmk}

\begin{qst}
Is the monoid of compact elements in the Cuntz semigroup of a real rank zero \ca{} a refinement monoid?
\end{qst}

Recall from \cite[Section~3]{Ror92StructureUHF2} that a \CuSgp{} $S$ is \emph{almost unperforated} if, given $x,y\in S$, we have $x\leq y$ whenever $(k+1)x\leq ky$ for some $k\in\NN$.

\begin{lma}
\label{prp:V-AlmUnp-pre}
Let $S$ be a \CuSgp{} satisfying \axiomO{5}-\axiomO{8}.
Assume that $S$ is almost unperforated and that $S\otimes \{ 0,\infty \}$ is algebraic.
Then~$S$ has property~(V).
\end{lma}
\begin{proof}
To verify property (V), let $c,d_j',d_j,x\in S$ satisfy $d_j'\ll d_j$ for $j=1,2$, and 
\[
d_1,d_2 \ll c, \andSep
c+d_1,c+d_2\ll x.
\] 

Choose $c'',c'\in S$ such that
\[
d_1,d_2 \ll c'' \ll c' \ll c.
\]

Using that $S\otimes \{ 0,\infty \}$ is algebraic, we may apply \autoref{prp:CharLatAlgbraic} to ensure that $c'\ll \infty c''$.
Let $k\in\NN$ be such that $c'\leq kc''$.
Applying \axiomO{5} for $c'' \ll c'\leq x$, we find~$f$ such that 
\[
c''+f \leq x \leq c'+f.
\]

Then
\[
d_1+c', d_2+c' \leq x \leq c'+f.
\]

It follows that
\[
(d_1+d_2)+c'
\leq d_1+c'+f
\leq c'+2f,
\]
and then
\begin{align*}
2(k+1)(d_1+d_2)
&\leq (2k+2)(d_1+d_2)+c' \\
&\leq (2k+1)(d_1+d_2)+c'+2f \\
&\leq (2k)(d_1+d_2)+c'+4f \\
&\leq \ldots 
\leq c'+4(k+1)f
\leq kc''+\infty f.
\end{align*}

Thus,
\[
2(k+1) (d_1+d_2) 
\leq kc''+\infty f.
\]

Since $S$ satisfies \axiomO{8}, \cite[Proposition~7.8]{ThiVil21arX:NowhereScattered} allows us to find $g\in S$ satisfying
\[
(k+1)(2g) \ll kc'', \andSep 
d_1'+d_2' \ll g+\infty f.
\]

Using that $S$ is almost unperforated, we obtain $2g\ll c''$. 
Let $g'\in S$ satisfy $g'\ll g$ and $d_1'+d_2' \ll g'+\infty f$. 
Applying \autoref{prp:DivO5}, we find $y\in S$ with 
\[
g'+y \leq c'' \leq 2y, \andSep
g'\ll y.
\]

Set $z:=g'+f$.
Then
\[
y+z = y+g'+f \leq c''+f \leq x.
\]

Further,
\[
d_1'+d_2' \leq 2c'' \leq 4y, \andSep
d_1'+d_2' \leq g'+\infty f \leq \infty(g'+f)=\infty z. \qedhere
\]
\end{proof}

A \ca{} $A$ is said to have \emph{strict comparison of positive elements} if for all $a,b \in (A\otimes\KK)_+$, we have $a \precsim b$ whenever $a$ belongs to the closed ideal generated by~$b$ and $d_\tau(a) < d_\tau(b)$ for every lower semicontinuous, $[0,\infty]$-valued $2$-quasitraces~$\tau$ on~$(A\otimes\KK)_+$ with $d_\tau(b)=1$.
It follows from \cite[Proposition~6.2]{EllRobSan11Cone} that a \ca{} $A$ has strict comparison of positive elements if and only if $\Cu(A)$ is almost unperforated.
(For simple \ca{s}, this was shown in \cite[Corollary~4.6]{Ror04StableRealRankZ}.)

\begin{thm}
\label{prp:V-AlmUnp}
Let $A$ be a separable \ca{} with strict comparison of positive elements and with topological dimension zero.
Then~$\Cu(A)$ has property~(V).
\end{thm}
\begin{proof}
The Cuntz semigroup of every \ca{} satisfies \axiomO{5}-\axiomO{8}. 
By \cite[Proposition~6.2]{EllRobSan11Cone}, $\Cu(A)$ is almost unperforated.
Further, it follows from \autoref{prp:CharTopDimZero} that $\Cu(A)\otimes\{0,\infty\}$ is algebraic.
Thus, we may apply \autoref{prp:V-AlmUnp-pre} to deduce that $\Cu(A)$ has property~(V).
\end{proof}

\section{Divisibility vs. weak divisibility}
\label{sec:Div}

Let $S$ be a weakly $(2,\omega)$-divisible \CuSgp{}.
In this section we investigate conditions ensuring that $S$ is automatically $(2,\omega)$-divisible.
By definition of weak $(2,\omega)$-divisibility, for every $x'\ll x$ in $S$ there are elements $y_1,\ldots,y_n$ such that
\[
2y_1,\ldots,2y_n\leq x, \andSep x'\ll \infty(y_1+\ldots+y_n).
\]

If this always holds with $n=1$, then $S$ is $(2,\omega)$-divisible.
We show that we can always accomplish the above with $n=2$ if $S$ is ideal-filtered and satisfies \axiomO{5}-\axiomO{8};
see \autoref{prp:WkDivDownTo2}.
Using this, we show that nowhere scattered \ca{s} with ideal-filtered Cuntz semigroup `almost' have the Global Glimm Property;
see \autoref{prp:FI-PartialGGP}.

If, additionally, $S$ has property~(V), then we can further reduce $n$ to $1$, thus establishing the main technical result of the paper, \autoref{prp:MainThmCu}.

\begin{lma}
\label{prp:WkDivFirstStep}
Let $S$ be an ideal-filtered \CuSgp{} satisfying \axiomO{5}, and let $x\in S$ be weakly $(2,\omega)$-divisible.
Then for every $x'\in S$ satisfying $x'\ll x$, there exist $c,d_1,\ldots,d_n\in S$ such that
\[
c+d_1, \ldots, c+d_n \leq x, \quad
x'\ll\infty c, \andSep
x'\ll d_1+\ldots+d_n.
\]
\end{lma}
\begin{proof}
Using that $x$ is weakly $(2,\omega)$-divisible, we obtain $y_1,\ldots,y_n\in S$ such that
\[
2y_1,\ldots,2y_n\ll x, \andSep
x'\ll y_1+\ldots+y_n.
\]

Choose $y_j'\ll y_j$ in $S$ satisfying $x'\ll y_1'+\ldots+y_n'$. Using \autoref{prp:DivO5} for each $j\leq n$, we find $c_j\in S$ such that
\[
y_j'+c_j \leq x \leq 2c_j, \andSep y_j' \ll c_j.
\]

Since $S$ is ideal-filtered, we obtain $c\in S$ such that
\[
x' \ll \infty c, \andSep
c\ll c_1,\ldots,c_n.
\]

Then $c$ and $d_j:=y_j'$ have the claimed properties.
\end{proof}

\begin{prp}
\label{prp:WkDivDownTo2}
Let $S$ be an ideal-filtered \CuSgp{} satisfying \axiomO{5}-\axiomO{8}, and let $x\in S$ be weakly $(2,\omega)$-divisible.
Then for every $x'\in S$ satisfying $x'\ll x$, and every $m\in\NN$, there exist $c,d_1,d_2\in S$ such that
\[
d_1,d_2 \ll c, \quad
c+d_1,c+md_2 \ll x, \quad
x'\ll\infty c, \andSep
x'\ll \infty(d_1+d_2).
\]

In particular, we have $2d_1,md_2\leq x$ and $x' \ll \infty(d_1+d_2)$.
\end{prp}
\begin{proof}
Let $x'\in S$ satisfy $x'\ll x$, and let $m\in\NN$.
We may assume that $m\geq 1$.
Choose $x^{\natural\natural},x^\natural$ such that $x' \ll x^{\natural\natural} \ll x^\natural \ll x$.
Applying \autoref{prp:WkDivFirstStep}, we obtain $y,z_1,\ldots,z_n$ such that
\[
y+z_1, \ldots, y+z_n \leq x, \quad
x^\natural \ll \infty y, \andSep
x^\natural \ll z_1+\ldots+z_n.
\]

Choose $y'\in S$ such that
\[
x^\natural \ll \infty y', \andSep
y' \ll y.
\]

Applying \axiomO{5} for $y'\ll y\leq x$, we obtain $f\in S$ such that
\[
y'+f\leq x\leq y+f.
\]

For each $j$, we have
\[
z_j + y \leq x \leq y+f,
\]
and, as in the proof of \autoref{prp:V-AlmUnp-pre}, we get
\begin{align*}
(m+1)x^\natural
\ll (m+1)(z_1+\ldots+z_n) 
\leq y+\infty f.
\end{align*}

Applying \cite[Proposition~7.8]{ThiVil21arX:NowhereScattered}, we obtain $g\in S$ such that
\[
(m+1)g \ll y, \andSep 
x^{\natural\natural} \ll g+\infty f.
\]

Choose $g'\in S$ such that $g'\ll g$ and $x^{\natural\natural} \ll g'+\infty f$. It follows from \autoref{prp:DivO5} that there exists $h\in S$ with 
\[
mg'+h \leq y \leq (m+1)h, \andSep
g'\ll h
\]
and, consequently, we have
\[
x^{\natural\natural} \ll x^\natural \ll \infty y \leq \infty h, \andSep
x^{\natural\natural} \ll x^\natural \ll \infty y'.
\]

Applying that $S$ is ideal-filtered, we obtain $c\in S$ such that
\[
x^{\natural\natural} \ll \infty c, \andSep
c \ll h,y'.
\]

We have
\[
x' 
\ll x^{\natural\natural}
\ll g'+\infty f.
\]

Applying \axiomO{6}, we obtain $s,t\in S$ such that
\[
x' \ll s+t, \quad
s \leq x^{\natural\natural},g', \andSep
t \leq x^{\natural\natural},\infty f.
\]

Choose $s',t'\in S$ such that
\[
x' \ll s'+t', \quad
s' \ll s, \andSep
t' \ll t.
\]

We have
\[
t' \ll t \leq \infty f, \andSep
t' \ll t \leq x^{\natural\natural} \leq \infty c.
\]

Applying again that $S$ is ideal-filtered, we obtain $d_1$ such that
\[
d_1 \ll f,c, \andSep
t' \ll\infty d_1.
\]

Similarly, since we also have
\[
s' \ll s \leq x^{\natural\natural} \leq \infty c,
\]
we can apply that $S$ is ideal-filtered to obtain $d_2\in S$ such that
\[
d_2 \ll s,c, \andSep
s' \ll \infty d_2.
\]

Let us verify that $c,d_1$ and $d_2$ have the desired properties.
First, we clearly have $d_1,d_2\ll c$ and $x' \ll \infty c$.
Further, we have
\[
c + d_1 
\leq c + f 
\leq y'+f 
\leq x, \andSep
c+md_2
\leq h+ms
\leq h+mg' 
\leq y
\leq x.
\]

Lastly, we have
\[
x' 
\ll s'+t'
\ll \infty d_2 + \infty d_1
= \infty (d_1+d_2). \qedhere
\]
\end{proof}

\begin{thm}
\label{prp:MainThmCu}
A \CuSgp{} satisfying \axiomO{5}-\axiomO{8} is $(2,\omega)$-divisible if and only if it is weakly $(2,\omega)$-divisible, ideal-filtered, and has property~(V).
\end{thm}
\begin{proof}
The forward implication follows from Propositions~\ref{prp:DivImplIdealFiltered} and~\ref{prp:DivImplV}.

To show the converse, let $S$ be a weakly $(2,\omega)$-divisible, ideal-filtered \CuSgp{} with property~(V), and let $x'\ll x$ in $S$. 
Take $x''$ with $x'\ll x''\ll x$. 
Applying \autoref{prp:WkDivDownTo2} for $m=1$, we obtain $c, d_1, d_2$ such that 
\[
d_1,d_2 \ll c, \quad
c+d_1,c+d_2 \ll x, \quad 
x'' \ll \infty c, \andSep 
x'' \ll \infty (d_1 + d_2).
\]

Choose $d_1',d_2'\in S$ such that
\[
x'' \ll \infty (d_1' + d_2'), \quad
d_1' \ll d_1, \andSep
d_2' \ll d_2.
\]

Applying property~(V), we get elements $z,y$ such that 
\[
y+z \leq x, \andSep
d_1' + d_2' \leq \infty y,\infty z.
\]
 
It follows that
\[
x' \ll x'' \ll \infty (d_1' + d_2') \leq \infty y,\infty z.
\]

Using that $S$ is ideal-filtered, we find $u$ such that 
\[
u \ll y,z, \andSep
x' \ll \infty u.
\]

Thus, one has 
\[
2u \leq y+z \leq x, \andSep 
x'\ll \infty u,
\]
as desired.
\end{proof}

\begin{cor}
\label{prp:CharDivRSF}
A residually stably finite \CuSgp{} satisfying \axiomO{5}-\axiomO{8} is $(2,\omega)$-divisible if and only if it is weakly $(2,\omega)$-divisible and ideal-filtered.
\end{cor}
\begin{proof}
Let $S$ be a residually stably finite \CuSgp{} satisfying \axiomO{5}-\axiomO{8}.
By \autoref{prp:V-RSF-pre}, $S$ automatically satisfies~(V).
Therefore, the result follows from \autoref{prp:MainThmCu}.
\end{proof}

\section{The Global Glimm Problem}
\label{sec:GlimmProblem}

In this section, we prove the main result of the paper:
A \ca{} has the Global Glimm Property if and only if it is nowhere scattered and its Cuntz semigroup is ideal-filtered and has property~(V);
see \autoref{prp:CharGGP}.
This turns the Global Glimm Problem into the question of  whether ideal-filteredness and property~(V) automatically hold for Cuntz semigroups of nowhere scattered \ca{s};
see \autoref{pbm:AutomaticIFV}.

In Sections~\ref{sec:IdealFiltered} and \ref{sec:V} we have shown that ideal-filteredness and property~(V) automatically hold for certain classes of \ca{s} (independent of whether the \ca{s} are nowhere scattered).
Using this, we recover the known solutions of the Global Glimm Problem for \ca{s} of stable rank one (\autoref{prp:GGP-SR1}) and real rank zero (\autoref{prp:GGP-RR0}).
We also obtain a solution for separable \ca{s} with topological dimension zero that are also residually stably finite or have strict comparison of positive elements;
see \autoref{prp:GGP-TopDimZero}.

Further, we show that every \ca{} contains a (unique) largest ideal that has the Global Glimm Property;
see \autoref{prp:largestGlimmIdeal}.

\begin{thm}
\label{prp:CharGGP}
Let $A$ be a \ca{}. 
Then the following are equivalent:
\begin{enumerate}
\item 
$A$ has the Global Glimm Property;
\item
$\Cu(A)$ is $(2,\omega)$-divisible;
\item
$\Cu(A)$ is weakly $(2,\omega)$-divisible, ideal-filtered and has property~(V);
\item
$A$ is nowhere scattered and $\Cu(A)$ is ideal-filtered and has property~(V).
\end{enumerate}
\end{thm}
\begin{proof}
The equivalence between~(1) and~(2) follows from \autoref{prp:GlimmFirstChar}.
Since the Cuntz semigroup of a \ca{} satisfies \axiomO{5}-\axiomO{8}, the equivalence between~(2) and~(3) follows from \autoref{prp:MainThmCu}.
Finally, by \cite[Theorem~8.9]{ThiVil21arX:NowhereScattered}, a \ca{} is nowhere scattered if and only if its Cuntz semigroup is weakly $(2,\omega)$-divisible, which proves the equivalence between~(3) and~(4).
\end{proof}

\begin{pbm}[Reformulation of the Global Glimm Problem]
\label{pbm:AutomaticIFV}
Let $A$ be a nowhere scattered \ca.
Is $\Cu(A)$ ideal-filtered and has property~(V)?
\end{pbm}

We recover the positive solution to the Global Glimm Problem for \ca{s} of stable rank one from \cite{AntPerRobThi22CuntzSR1}:

\begin{prp}
\label{prp:GGP-SR1}
Let $A$ be a \ca{} of stable rank one.
Then $A$ has the Global Glimm Property if and only if $A$ is nowhere scattered.
\end{prp}
\begin{proof}
By \autoref{prp:IdealFilteredSR1}, $\Cu(A)$ is ideal-filtered.
Further, every \ca{} of stable rank one is residually stably finite, and therefore $\Cu(A)$ has property~(V) by \autoref{prp:V-RSF}.
Therefore, the result follows from \autoref{prp:CharGGP}.
\end{proof}

We also recover the positive solution to the Global Glimm Problem for \ca{s} of real rank zero from \cite{EllRor06Perturb}:

\begin{prp}
\label{prp:GGP-RR0}
Let $A$ be a \ca{} of real rank zero.
Then $A$ has the Global Glimm Property if and only if $A$ is nowhere scattered.
\end{prp}
\begin{proof}
By Theorems~\ref{prp:IdealFilteredRR0} and~\ref{prp:V-rr0}, $\Cu(A)$ is ideal-filtered and has property~(V).
Therefore, the result follows from \autoref{prp:CharGGP}.
\end{proof}

Additionally, we obtain a positive solution to the Global Glimm Problem for new classes of \ca{s}:

\begin{prp}
\label{prp:GGP-TopDimZero}
Let $A$ be a separable \ca{} with topological dimension zero. 
Assume that $A$ is residually stably finite or has strict comparison of positive elements.
Then $A$ has the Global Glimm Property if and only if $A$ is nowhere scattered.
\end{prp}
\begin{proof}
By \autoref{prp:IdealFilteredTopDimZero}, $\Cu(A)$ is ideal-filtered.
Using \autoref{prp:V-RSF} if $A$ is residually stably finite, or \autoref{prp:V-AlmUnp} if $A$ has strict comparison of positive elements, we see that $\Cu(A)$ has property~(V).
Hence, the result follows from \autoref{prp:CharGGP}.
\end{proof}

If we only know that $\Cu(A)$ is ideal-filtered (for example, if $A$ is separable and has topological dimension zero), then we obtain the following `almost solution' of the Global Glimm Problem:

\begin{prp}
\label{prp:FI-PartialGGP}
Let $A$ be a nowhere scattered \ca{} such that $\Cu(A)$ is ideal-filtered.
Then for each $a\in A_+$, each natural number $m\geq 2$, and each $\varepsilon>0$, there exist \stHom{s} $\varphi_1\colon C_0((0,1],M_2)\to\overline{aAa}$ and $\varphi_2\colon C_0((0,1],M_m)\to\overline{aAa}$ such that $(a-\varepsilon)_+$ belongs to the ideal of $A$ generated by the combined images of $\varphi_1$ and $\varphi_2$.

In particular, there exist \emph{two} \stHom{s} $C_0((0,1],M_2)\to\overline{aAa}$ such that $(a-\varepsilon)_+$ belongs to the ideal of $A$ generated by their combined images.
\end{prp}
\begin{proof}
As mentioned in \autoref{sec:Prelims}, $\Cu(A)$ is a \CuSgp{} satisfying \axiomO{5}-\axiomO{8} and, using \cite[Theorem~8.9]{ThiVil21arX:NowhereScattered}, we know that $\Cu(A)$ is weakly $(2,\omega)$-divisible.
Further, by assumption, $\Cu(A)$ is ideal-filtered.

Now let $a\in A_+$ and $\varepsilon >0$. 
Set $x=[a]$ and $x'=[(a-\varepsilon)_+]$. 
Applying \autoref{prp:WkDivDownTo2} to $x'\ll x$, one obtains elements $d_1,d_2\in\Cu(A)$ such that
\[
2d_1, md_2 \leq x, \andSep
x' \ll \infty(d_1+d_2).
\]

Take $d_1',d_2'\in \Cu (A)$ satisfying 
\[
x' \ll \infty(d_1'+d_2'), \quad
d_1' \ll d_1, \andSep
d_2' \ll d_2.
\]

Since $2d_1\leq x$, \cite[Lemma~2.4]{RobRor13Divisibility} implies that there exists a \stHom{} $\varphi_1\colon C_0((0,1],M_2)\to\overline{aAa}$ with $d_1'\ll [\varphi_1 (e_{1,1}\otimes\iota)]\ll d_1$, where $e_{1,1}\otimes\iota$ denotes the map $t\mapsto te_{1,1}$ in $C_0((0,1],M_2)$.
Similarly, using that $md_2 \leq x$ and \cite[Lemma~2.4]{RobRor13Divisibility}, we find another \stHom{} $\varphi_2\colon C_0((0,1],M_m)\to\overline{aAa}$ with $d_2'\ll [\varphi_2 (e_{1,1}\otimes\iota)]\ll d_2$.

Using that $[(a-\varepsilon)_+]=x' \leq \infty(d_1'+d_2')$, it follows that $(a-\varepsilon)_+$ belongs to the ideal generated by the combined images of $\varphi_1$ and $\varphi_2$, as desired.
\end{proof}

The next result shows that every \ca{} contains a (unique) largest ideal that has the Global Glimm Property.

\begin{thm}
\label{prp:largestGlimmIdeal}
Let $A$ be a \ca.
Then there exists an ideal $I_{\mathrm{Glimm}}(A)$ of~$A$ such that $I_{\mathrm{Glimm}}(A)$ has the Global Glimm Property, and such that every ideal $I\subseteq A$ with the Global Glimm Property satisfies $I\subseteq I_{\mathrm{Glimm}}(A)$.
\end{thm}
\begin{proof}
Let $\mathcal{J}$ denote the collection of ideals $I\subseteq A$ such that $I$ has the Global Glimm Property.
To show that $\mathcal{J}$ is upward directed, let $I,J\in\mathcal{J}$.
We consider the ideal $I+J\subseteq A$.
Viewing $I$ as an ideal of $I+J$, it follows that $(I+J)/I$ is isomorphic to a quotient of $J$.
Applying \autoref{prp:GlimmExtension}, we see that $(I+J)/I$ has the Global Glimm Property and, consequently, so does $I+J$.

Set $J:=\overline{\bigcup\mathcal{J}}$.
Then $J$ is an ideal of $A$ that is the inductive limit of the inductive system $\mathcal{J}$ (indexed over itself).
By \autoref{prp:Limits}, $J$ has the Global Glimm Property.
Then $I_{\mathrm{Glimm}}(A):=J$ has the claimed properties.
\end{proof}

\begin{prp}
Let $A$ be a nowhere scattered \ca{} with finite primitive ideal space.
Then $A$ has the Global Glimm Property.
\end{prp}
\begin{proof}
First, note that a simple, nowhere scattered \ca{} $B$ always satisfies the Global Glimm Property. 
Indeed, $B$ and all of its hereditary sub-\ca{s} are non-type I by \cite[Theorem~3.1]{ThiVil21arX:NowhereScattered}. Thus, it follows from Glimm's classical theorem that each hereditary sub-\ca{} of $B$ contains a sub-\ca{} isomorphic to $M_2 (C_0 ((0,1]))$; see \cite[Proposition~4.10]{KirRor00PureInf}. This implies that every hereditary sub-\ca{} has a full square-zero element. By \autoref{prp:GlimmFirstChar}, $B$ has the Global Glimm Property.

Let us now prove the statement by induction on $\vert \Prim (A)\vert$, the number of points in the primitive ideal space $\Prim (A)$ of the \ca{} $A$.
 
If $\vert \Prim (A)\vert =1$, the result holds by the previous argument. 
Thus, fix $n\in\NN$ and assume that we have proven the result for all nowhere scattered \ca{s}~$B$ with $\vert \Prim (B)\vert\leq n$.

Let $A$ be a nowhere scattered \ca{} such that $\vert \Prim (A)\vert =n+1$, and take a proper, nonzero ideal $J\subseteq A$. 
Then, both $\Prim (A/J)$ and $\Prim (J)$ have strictly fewer points than $\Prim (A)$. 
Using our induction hypothesis, we see that $A/J$ and $J$ satisfy the Global Glimm Property.
Now it follows from \autoref{prp:GlimmExtension} that $A$ has the Global Glimm Property.
\end{proof}

\begin{rmks}
(1)
Following \cite[2.1(iv)]{BroPed09Limits}, we say that a \ca{} has \emph{generalized stable rank one} if it has a composition series $(I_\lambda)_{\lambda\leq\kappa}$ of ideals such that each successive quotient $I_{\lambda+1}/I_\lambda$ has stable rank one.
Generalized real rank zero is defined analogously.

It follows from \autoref{prp:GlimmExtension},  \autoref{prp:Limits} and \autoref{prp:GGP-SR1}, that every nowhere scattered \ca{} of generalized stable rank one has the Global Glimm Property.
Similarly, using \autoref{prp:GGP-RR0} instead of \autoref{prp:GGP-SR1}, we see that every nowhere scattered \ca{} of generalized real rank zero has the Global Glimm Property.

(2)
By \cite[Proposition~4.2]{Ror88AdvUnitaryRank} and \cite[Theorem~2.3]{BroPed07IdealStructure}, every \ca{} $A$ contains a largest ideal $I_{\mathrm{sr1}}(A)$ that has stable rank one, and a largest ideal $I_{\mathrm{rr0}}(A)$ that has real rank zero.
Assuming that $A$ is nowhere scattered, it follows that $I_{\mathrm{sr1}}(A)\subseteq I_{\mathrm{Glimm}}(A)$ and $I_{\mathrm{rr0}}(A)\subseteq I_{\mathrm{Glimm}}(A)$.
More generally, $I_{\mathrm{Glimm}}(A)$ contains the largest ideal of $A$ that has generalized stable rank one (generalized real rank zero).
\end{rmks}



\providecommand{\href}[2]{#2}

\end{document}